\documentclass[a4paper,12pt,]{article}
\usepackage{hyperref}
\usepackage{amsfonts}
\usepackage{amsmath}
\usepackage{amssymb}
\usepackage{amsthm}
\usepackage{enumerate}
\usepackage[mathscr]{eucal}
\usepackage{eqlist}
\usepackage{graphicx}
\usepackage{esint}
\usepackage{color}
\usepackage{booktabs} 
\usepackage{multirow} 

\usepackage[body={15.6true cm, 24.9true cm}]{geometry}

\newtheorem{theorem}{\hskip\parindent\bf Theorem}[section]
\newtheorem{lemma}{\hskip\parindent\bf Lemma}[section]
\newtheorem{proposition}{\bf Proposition}[section]
\newtheorem{remark}{\hskip\parindent\bf Remark}[section]
\newtheorem{definition}{\hskip\parindent\bf Definition}[section]
\numberwithin{equation}{section}
\numberwithin{equation}{section} \allowdisplaybreaks
\usepackage{mathdots}

\renewcommand\abstract{{\bf Abstract}}

\begin{document}

\title {\LARGE \bf Far field refraction problem with loss of energy in negative refractive index material}

\author{{Haokun Sui$^{a}$,~~Feida Jiang$^{a,b,}$\thanks{Corresponding author.
E-mail address: jiangfeida@seu.edu.cn}}}
\maketitle
\begin{center}
\begin{minipage}{12cm}
\begin{description}
\item \small
 $a$.~School of Mathematics and Shing-Tung Yau Center of Southeast University, Southeast University, Nanjing 211189, P.R. China
 \item \small
$b$.~Shanghai Institute for Mathematics and Interdisciplinary Sciences, Shanghai 200433, P.R. China
\end{description}
\end{minipage}
\end{center}


	

\begin{abstract}{\bf:}{\footnotesize
~This paper studies the far field refraction problem in negative refractive index material with loss of energy, which is a remaining problem in E. Stachura, Nonlinear Anal. 2017;157:76-103. The analysis is divided into two cases according to the relative refractive index $\kappa$, that is, $\kappa<-1$ and $-1<\kappa<0$. For each case, we use the Minkowski method to establish the existence of the weak solution when the target measure is either discrete or a finite Radon measure. Eventually, the inequality involving a Monge-Amp\`ere type operator satisfied by the solution of the problem is derived, which is useful to understand this complex optical phenomenon.}		
\end{abstract}
	
	

\tableofcontents

\section{Introduction}\label{Section 1}

\subsection{Background}

Negative refraction was first proposed by the Russian scientist Veselago in 1968 \cite{Ve68}, referring to the phenomenon that when a light wave is incident from a material with a positive refractive index to the interface of a material with a negative refractive index, the refraction of the light wave is opposite to conventional refraction, with the incident and refracted waves located on the same side of the interface normal. Negative refractive index materials are artificially structured materials with both permittivity $\epsilon$ and permeability $\mu$ are negative. In such materials, the electric vector, magnetic vector, and wave vector of electromagnetic waves form a left-handed system, hence they are also called ``left-handed material''. While traditional materials have positive refractive indices, the unique properties of negative refractive index materials give them disruptive potential in fields like optics and electromagnetism.

In 2000, Smith et al. \cite{SP00} artificially synthesized the world's first medium with negative equivalent permittivity and permeability in the microwave range using copper-based composite materials. In~2001, Shelby et al. \cite{SS01} made a prism from existing negative refractive index materials, experimentally confirming negative refraction and showing that light incident on a negative refractive index medium surface refracts to the same side of the interface normal as the incident light. Since the beginning of 21st century, negative refractive index materials have been widely used in optical invisibility \cite{CT05}, perfect lens imaging \cite{PR02,Pen00}, wireless directional radiation \cite{ET02} and manufacturing of novel optical devices such as high-capacity optical discs \cite{LH04}.

In recent years, near field refraction and far field refraction problems have been widely researched mathematically. Near field refraction problem refers that given two media \textsc{I} and \textsc{II}, and a light source in medium \textsc{I}, constructing a refracting surface $\mathcal{R}$ separating media \textsc{I} and \textsc{II}, such that all ray emitted from the light source refract through $\mathcal{R}$ to a specified point $P$ in medium \textsc{II}. In 2014, Gutiérrez and Huang \cite{GH14} studied the single surface near field refraction problem in positive refractive index media without loss of energy. The existence of weak solutions to the refraction problem was proved by using Minkowski method and the corresponding partial differential equation was also derived. For related research on near field refraction problem, see \cite{GM21,GSA18, GT14, GT15, GT19}. Far field refraction problem refers that given two media \textsc{I} and \textsc{II}, and a light source in medium \textsc{I}, constructing a refracting surface $\Gamma$ separating media \textsc{I} and \textsc{II}, such that all ray emitted from the light source refract through $\Gamma$ to a specified direction $m$ in medium \textsc{II}. In 2009, Gutiérrez and Huang \cite{GH09} studied the single surface far field refraction problem in positive refractive index media without loss of energy. They used the optimal transportation method to prove the existence of weak solutions of this problem, derived the corresponding partial differential equation and verified that the equation satisfies the A3 condition in \cite{MTW05}. In 2017, Gutiérrez and Sabra \cite{GS18} studied the double surface far field refraction problem in positive refractive index media without loss of energy. They proved that given a lower surface, there exists an upper surface that satisfies the Monge-Amp\`ere type equation which can refract parallel light in a given direction. There are also some other research on far field refraction, see \cite{AF24,DG17,GH19,GT11}.

In 2015, Stachura and Gutiérrez \cite{GS15} first studied the refraction problem in negative refractive index media mathematically. They proposed the Snell's law for negative refractive index material and extending the near field and far field refraction problems from positive refractive index media to negative refractive index media. Later in 2016, Stachura and Gutiérrez \cite{GS16} further generalized their previous work by investigating double refraction in both near and far fields in negative refractive index media. In 2017, Stachura \cite{St17} conducted a deeper study of the refraction problem in negative refractive index media. The Minkowski method was used to prove the existence of the weak solution of near field refraction problem, and the optimal transmission method was used to prove the existence of weak solution of far field refraction problem and derive the corresponding Monge-Amp\`ere type equation. However, these studies are based on the assumption of energy conservation. In fact, when the light ray emitted from medium \textsc{I} strikes the interface between medium \textsc{I} and \textsc{II}, it gives two rays, some of the ray will be refracted into medium \textsc{II}, while the other ray will be reflected back to medium \textsc{I}. Therefore, the energy of the incident light ray is not equal to that of the refracted light ray. In 2013, Mawi and Gutiérrez \cite{GM13} studied the far field refraction problem with loss of energy in positive refractive index media, showing the existence of the weak solution of far field refraction problem with loss of energy and deriving the corresponding Monge-Amp\`ere type equation. In fact, refraction with loss of energy also occurs in negative refractive index material \cite{DC07}. Figure \ref{fig1} shows the refraction problem with loss of energy in negative refractive index material, indicating that when an incident light ray having direction of propagation $x\in S^{n-1}$ strikes at $\Gamma$, it will split into two rays: a reflected ray in direction $r\in S^{n-1}$ back into medium \textsc{I} and a refracted ray in direction $m\in S^{n-1}$ transmitted into medium \textsc{II}. 

\begin{figure}[h]
  \centering
  \includegraphics[width=8.3cm]{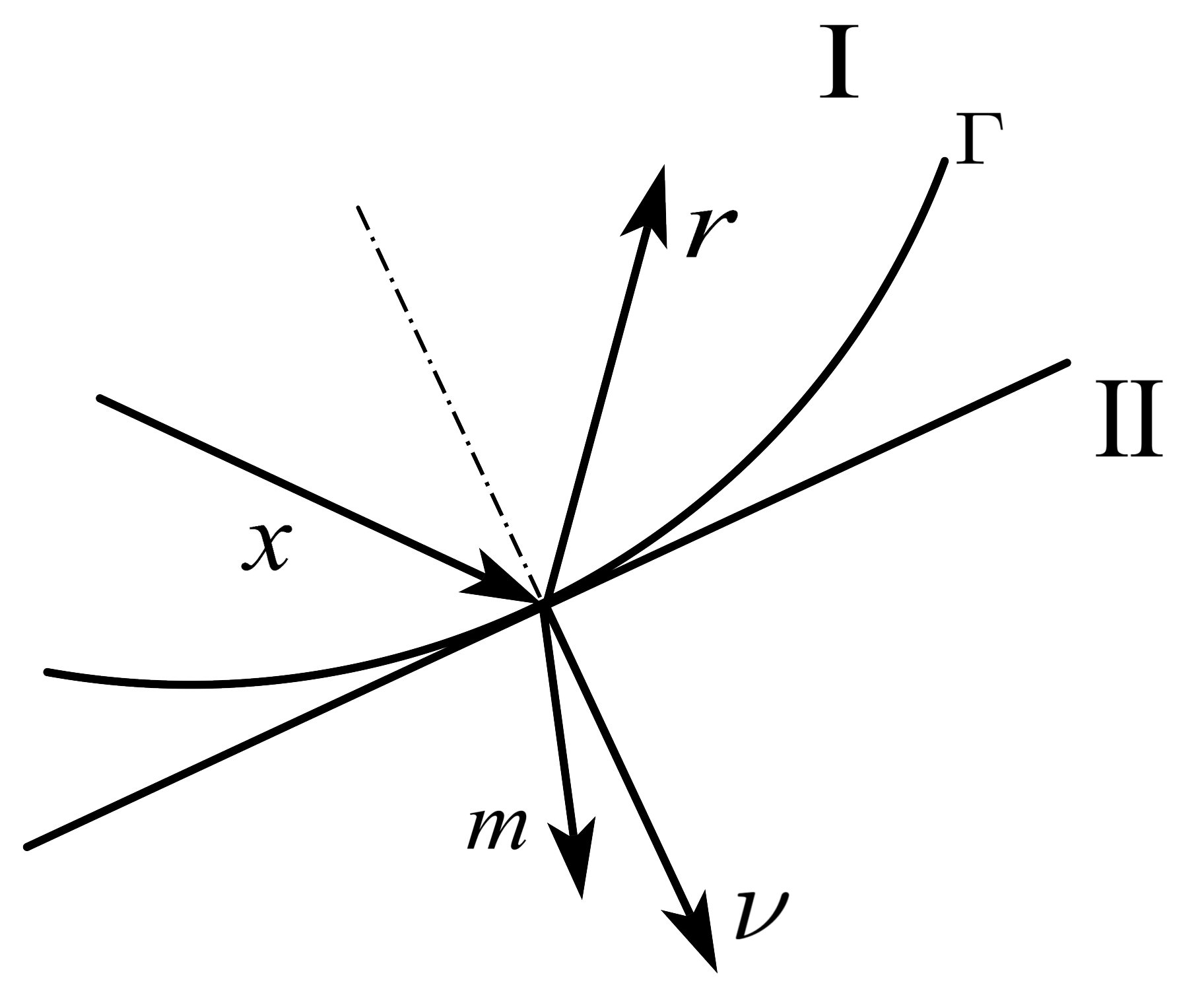}
\caption{Sketch of the refraction problem with loss of energy in negative refractive index material.}\label{fig1}
\end{figure}

\subsection{Description of the problem}

In this paper, we consider the following problem: Suppose that $\Omega$ and $\Omega^{\ast}$ are two domains in $S^{n-1}$, $f$ and $g$ are two integrable functions on $\Omega$ and $\Omega^{\ast}$ respectively, that is, $f\in L^{1}(\bar{\Omega})$, $g\in L^{1}(\bar{\Omega}^{\ast})$. Consider two homogeneous, isotropic media: medium \textsc{I} and medium \textsc{II}, surrounded by  $\Omega$ and $\Omega^{\ast}$ respectively which have different optical densities. Given a direction $m\in \Omega^{\ast}$, we want to construct a surface $\Gamma$ separating media \textsc{I} and \textsc{II}, such that all rays emanate from the origin $O$, located in medium \textsc{I}, with directions $x \in \Omega$ and intensity $f(x)$, are refracted into medium \textsc{II}, with direction $m \in \Omega^{\ast}$ and intensity $g(m)$.  Assuming that the refractive index of medium \textsc{I} is $n_{1}>0$, the refractive index of medium \textsc{II} is $n_{2}<0$, and set the relative refractive index $\kappa= \dfrac{n_{2}}{n_{1}}$, so $\kappa<0$. Notice that in application, it is natural to study the refraction problem in $S^{2}$, see Figure \ref{fig2}. However, in this paper, we directly study the problem in $S^{n-1}$ with $n \geq 2$ for its generality.

\begin{figure}[h]
  \centering
  \includegraphics[width=8.3cm]{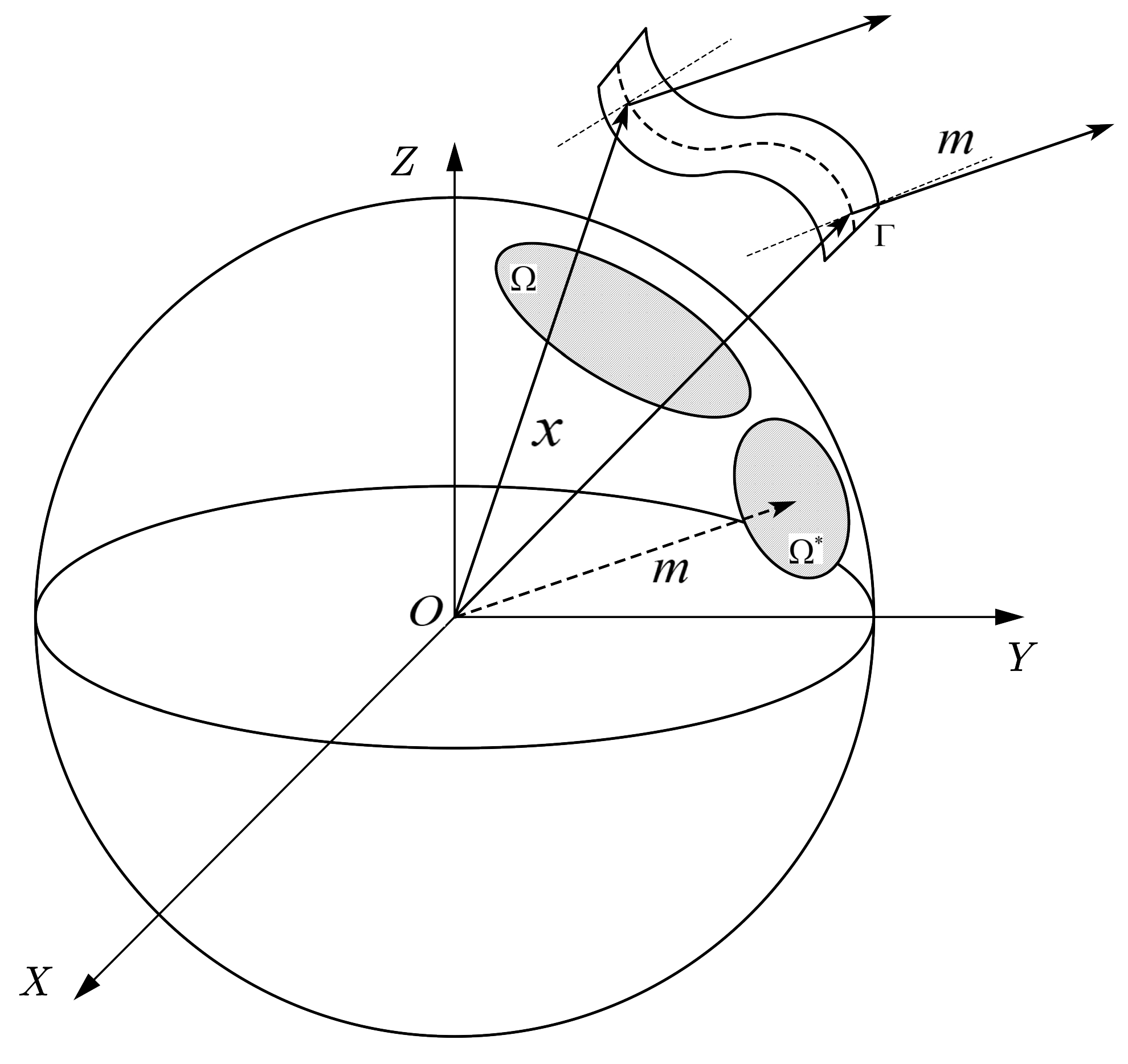}
\caption{Statement of the problem in $S^{2}$.}\label{fig2}
\end{figure}

\subsection{Main results}

There are three main methods for studying reflection and refraction problems \cite{Ma23}: The first method is using variational method to transform the problem into optimal transportation problem, the second method is using energy conservation conditions to derive the Monge-Amp\`ere type equation which the reflective or refractive surfaces satisfy, and the third method is Minkowski method. In this paper, we mainly use Minkowski method to study the far field refraction problem with loss of energy in negative refractive index medium. 

Minkowski method is an iterative approach for solving refraction and reflection problems in geometric optics. To sovle our problem, we first give some properties of refractor and Fresnel formula. Then we introduce the definition of the weak solution of the problem. Next, the existence results when the underlying measure is the finite sum of $\delta$-measures is proved by using approximation by hyperboloids or ellipsoids depending on whether $\kappa<-1$ or $-1<\kappa<0$, see {Theorems \ref{thm3.3} and \ref{thm4.2}}. Using these results, we prove the existence of the weak solution for the general finite Radon measure, see {Theorems \ref{thm3.4} and \ref{thm4.3}}.

Generally, refraction and reflection in geometric optics can be described by Monge-Amp\`ere type equation \cite{Ji23,JT14,JT18,LW21,Wa96}. Based on the definition of the weak solution of the far field refraction problem for the case $\kappa<-1$ and $-1<\kappa<0$, see {Definitions \ref{def3.4} and \ref{def4.4}}, the inequality involving a Monge-Amp\`ere type operator which the solution of the problem satisfies, see Theorem \ref{thm5.1}, is also derived in this paper. To the best of our knowledge, this work presents the first construction of a refractor in negative refractive index material that accounts for the energy used in internal reflection.

The rest of the content is organized as follows: In Section \ref{Section 2}, we give some preliminaries, namely Snell law and Fresnel formula in negative refractive index material. In Section \ref{Section 3}, we study the existence of the weak solution when $\kappa<-1$. We first study the existence of weak solutions in discrete situation, then use approximation by hyperboloids to investigate the existence of weak solutions in general situation. In Section \ref{Section 4}, we use similar way in Section \ref{Section 3} to study the existence of the weak solution when $-1<\kappa<0$. The inequality involving a Monge-Amp\`ere type operator which the solution of the problem satisfies is driven in Section \ref{Section 5}. Finally, in Section \ref{Section 6}, we summarize our work and compare it with previous research.

\section{Preliminaries}\label{Section 2}

\sloppy{}

In this section, we briefly introduce Snell law in vector form and Fresnel formula in negative refractive index material.

\subsection{Snell law in vector form}

Suppose $\Gamma$ is a surface in $\mathbb{R}^{n}$ that separates two homogeneous and isotropic media \textsc{I} and \textsc{II}, with refractive indices $n_{1}>0$ and $n_{2}<0$. A ray of light emitted from $O\in S^{n-1}$ in medium \textsc{I} with direction $x\in S^{n-1}$ strikes at $\Gamma$ at the point $P$, then the refracted ray has the direction $m\in S^{n-1}$ in medium \textsc{II}. Let $\nu$ be the unit normal to $\Gamma$ at $P$ going towards medium \textsc{II}, $\theta_{1}$ be the angel between $x$ and $\nu$ which called the angle of incidence and $\theta_{2}$ be the angle between $m$ and $\nu$ which called the angle of refraction. Then we have the well-known Snell law in scalar form:
\begin{equation}\label{2.1}
  n_{1}\sin\theta_{1} = n_{2}\sin\theta_{2}.
\end{equation}
This law can be written in vector form as:
\begin{equation}\label{2.2}
  n_{1}(x\times\nu) = n_{2}(m\times\nu).
\end{equation}

From (\ref{2.2}), it is easily seen that $x,m$ and $\nu$ are in the same plane. Since we have set $\kappa=\dfrac{n_{2}}{n_{1}}$, 
then~(\ref{2.2}) can be written as
\begin{equation}\label{2.3}
  x-\kappa m = \lambda \nu,
\end{equation}
where $\lambda\in \mathbb{R}$ is given by
\begin{equation}\label{2.4}
  \lambda = x\cdot \nu + \sqrt{(x\cdot \nu)^{2} - (1 - \kappa^{2})} = x\cdot \nu + \vert \kappa \vert \sqrt{1 - \kappa^{-2}(1 - (x\cdot \nu)^{2})}.
\end{equation}
If we set
\begin{equation}\label{2.5}
  \Phi(t) = t + \vert \kappa \vert \sqrt{1 - \kappa^{-2}(1 - t^{2})},
\end{equation}
then we have $\lambda = \Phi (x \cdot \nu)$.

Furthermore, we need to determine the physical constraints of $\bar{\Omega}$ and $\bar{\Omega}^{\ast}$ to ensure that total internal reflection cannot occur.

When $\kappa<-1$, that is, $n_{1}<\vert n_{2} \vert$, in this case, the direction of refracted ray $m$ is close to normal $\nu$. So when $\theta_{1} = \dfrac{\pi}{2}$, the angle of refraction attains its maximum $\theta_{2}^{\ast} = \arcsin\left(-\dfrac{1}{\kappa}\right) := \theta_{c}^{\ast}$. 
Then from Snell law~(\ref{2.1}), we have
\begin{equation*}
  \theta_{1} + \theta_{2} = \arcsin(-\kappa \sin\theta_{2}) + \theta_{2}.
\end{equation*}
Since the function $h(\theta) = \arcsin(-\kappa \sin\theta) + \theta$ is increasing on $[0,\theta_{c}^{\ast}]$, then we have $0\leq \theta_{1} + \theta_{2} = \dfrac{\pi}{2} + \theta_{c}^{\ast}$, so we have
\begin{equation*}
  x \cdot m = \cos(\theta_{1} + \theta_{2}) \geq \cos\left(\dfrac{\pi}{2} + \theta_{c}^{\ast}\right) = \frac{1}{\kappa}.
\end{equation*}

When $-1<\kappa<0$, that is, $n_{1}>|n_{2}|$, in this case, the direction of refracted ray $m$ is away from normal $\nu$. So when $\sin\theta_{1} = -\kappa = -\dfrac{n_{2}}{n_{1}}$, the angle of refraction attains its maximum $\theta_{2}^{\ast} = \dfrac{\pi}{2}$. Therefore, the sine value of the angle of incident is no larger than $-\kappa$, that is, $0\leq \theta_{1} \leq \theta_{c} := \arcsin(-\kappa)$. Then from Snell law (\ref{2.1}), we have
\begin{equation*}
   \theta_{1} + \theta_{2} = \arcsin\left(-\frac{1}{\kappa}\sin\theta_{1}\right) + \theta_{1}.
\end{equation*}
Since the function $h(\theta) = \arcsin\left(-\dfrac{1}{\kappa}\sin\theta\right) + \theta$ is increasing on $[0,\theta_{c}]$, then we have $0\leq \theta_{1} + \theta_{2} \leq \dfrac{\pi}{2} + \theta_{c}$, so we have
\begin{equation*}
  x \cdot m = \cos(\theta_{1} + \theta_{2}) \geq \cos\left(\frac{\pi}{2} + \theta_{c}\right) = \kappa.
\end{equation*}

From the above analysis, we have the following lemma:

\begin{lemma}
  Suppose that the refractive indices of media \textsc{I} and \textsc{II} are given by $n_{1}>0$ and $n_{2}<0$, and set $\kappa = \dfrac{n_{2}}{n_{1}}$.
    
    (a) If $\kappa <-1$, a light ray in medium \textsc{I} in the direction $x \in S^{n-1}$ is refracted by some surface into a light ray in medium \textsc{II} in the direction $m \in S^{n-1}$ if and only if $x\cdot m \geq \dfrac{1}{\kappa}$.
   
    (b) If $-1<\kappa <0$, a light ray in medium \textsc{I} in the direction $x \in S^{n-1}$ is refracted by some surface into a light ray in medium \textsc{II} in the direction $m \in S^{n-1}$ if and only if $x\cdot m \geq \kappa$.
\label{lem2.1}
\end{lemma}

\begin{remark}
  Lemma \ref{lem2.1} is typically used to solve the refraction problem without loss of energy. However, since this paper considers the refraction problem with loss of energy, we need to strengthen the conclusions of Lemma \ref{lem2.1} appropriately, see the following:
 
  Suppose that the refractive indices of media \textsc{I} and \textsc{II} surrounded by  $\Omega$ and $\Omega^{\ast}$ respectively, are given by $n_{1}>0$ and $n_{2}<0$, and set $\kappa = \dfrac{n_{2}}{n_{1}}$.
   
   (a) For the case $\kappa <-1$, we assume that there exists $\varepsilon >0$, such that
   \begin{equation}\label{2.6}
     \inf_{x \in \bar{\Omega}, m \in \bar{\Omega}^{\ast}} x \cdot m \geq \frac{1}{\kappa} + \varepsilon,
   \end{equation}
   then from Lemma \ref{lem2.1} (a), a light ray in medium \textsc{I} in the direction $x \in S^{n-1}$ can be refracted by some surface into a light ray in medium \textsc{II} in the direction $m \in S^{n-1}$.
   
   (b) For the case $-1<\kappa <0$, we assume that there exists $\varepsilon >0$, such that
   \begin{equation}\label{2.7}
     \inf_{x \in \bar{\Omega}, m \in \bar{\Omega}^{\ast}} x \cdot m \geq \kappa + \varepsilon,
   \end{equation}
   then from Lemma \ref{lem2.1} (b), a light ray in medium \textsc{I} in the direction $x \in S^{n-1}$ can be refracted by some surface into a light ray in medium \textsc{II} in the direction $m \in S^{n-1}$.
  \label{rem2.1}
\end{remark}


\subsection{Fresnel formula}

From the previous analysis,  we know that when the incident light ray strikes the surface $\Gamma$, it will split into refracted light ray and reflected light ray, so the energy of the incident light ray will be distributed to the refracted light ray and reflected light ray. This subsection briefly gives the energy distribution of reflected and refracted light ray  according to the electromagnetic field theory of light propagation.

Define $\mathbf{E} = \mathbf{E}(\mathbf{r},t)$ as electric field vector and $\mathbf{B} = \mathbf{B}(\mathbf{r},t)$ as magnetic field vector, where $\mathbf{r} = \mathbf{r}(x,y,z)$ represents a point in 3-d space and $t$ is the time, then we have the following system of Maxwell's equations absent from charges:
\begin{equation}\label{2.8}
  \left\{
   \begin{split}
     & \nabla\times \mathbf{E} = -\frac{\mu}{c} \frac{\partial \mathbf{B}}{\partial t},  \\
     & \nabla\times \mathbf{B} = -\frac{\epsilon}{c} \frac{\partial \mathbf{E}}{\partial t},  \\
     & \nabla \cdot (\epsilon \mathbf{E}) = 0,  \\
     & \nabla \cdot (\mu \mathbf{B}) = 0,
   \end{split}
    \right.
\end{equation}
where $c$ is the speed of light in vacuum, $\mu = \mu(x,y,z)$ is the magnetic permeability of the medium and $\epsilon = \epsilon(x,y,z)$ is the electric permittivity of the medium.

Assume that the waves are plane waves, that is, the waves have the same value at all points of any plane perpendicular to the direction of propagation, then from (\ref{2.8}), we have:
\begin{equation}\label{2.9}
  \left\{
  \begin{split}
      & \mathbf{E} = -\frac{c}{\epsilon \omega}(\mathbf{k} \times \mathbf{B}), \\
      & \mathbf{B} = \frac{c}{\mu \omega}(\mathbf{k} \times \mathbf{E}),
  \end{split}
  \right.
\end{equation}
where $c$ is the speed of light in free space, $\mathbf{k} = \displaystyle \frac{\omega}{v} \mathbf{s}$ represents the wave vector, $\omega$ represents the angular frequency of the electromagnetic wave, $v$ represents the speed of light ray in the medium and $\mathbf{s}$ is a unit vector.

The flow of the energy in an electromagnetic wave with electric field $\mathbf{E} = \mathbf{E}(\mathbf{r},t)$ and magnetic field $\mathbf{B} = \mathbf{B}(\mathbf{r},t)$ is given by Poynting vector
\begin{equation}\label{2.10}
  \mathbf{S} = \frac{c}{4\pi}\mathbf{E} \times \mathbf{B}.
\end{equation}
Then from (\ref{2.9}), we have
\begin{equation}\label{2.11}
  \mathbf{S} = \frac{c}{4\pi}\mathbf{E} \times (\frac{c}{\mu \omega} \mathbf{k} \times \mathbf{E}) = \frac{c}{4\pi} \sqrt{\frac{\epsilon}{\mu}}  \mathbf{E} \times \mathbf{s} \times \mathbf{E}.
\end{equation}

We denote quantities referring to the incident wave by the suffix $(i)$, to the refracted wave by $(t)$ and to the reflected wave by $(r)$. Choose a system of coordinates such that the normal $\nu$ to the interface $\Gamma$ at the point of incidence is on the $z$-axis and the $x$ and $y$ axes are on the plane perpendicular to $\nu$. So the tangent plane to $\Gamma$ at $P$ is the $xy$-plane and the incident plane is the $xz$-plane. Then each of the electric field and magnetic field vectors can be resolved into components parallel denoted by subscript $\parallel$ and perpendicular denoted by subscript $\bot$. Then we obtain:
\begin{equation*}
  \left\{
  \begin{split}
       & \mathbf{E}^{(\mathbf{i})}(\mathbf{r},t) = (-A_{\parallel}\cos\theta_{i} , A_{\bot} , A_{\parallel}\sin\theta_{i}) \cos(\omega(t - \frac{\mathbf{r} \cdot \mathbf{s^{(i)}}}{v_{1}})) = \mathbf{E}_{0}^{(\mathbf{i})} \cos(\omega(t - \frac{\mathbf{r} \cdot \mathbf{s^{(i)}}}{v_{1}})), \\
       & \mathbf{E}^{(\mathbf{t})}(\mathbf{r},t) = (-T_{\parallel}\cos\theta_{t} , T_{\bot} , T_{\parallel}\sin\theta_{t}) \cos(\omega(t - \frac{\mathbf{r} \cdot \mathbf{s^{(t)}}}{v_{2}})) = \mathbf{E}_{0}^{(\mathbf{t})} \cos(\omega(t - \frac{\mathbf{r} \cdot \mathbf{s^{(t)}}}{v_{2}})), \\
       & \mathbf{E}^{(\mathbf{r})}(\mathbf{r},t) = (-R_{\parallel}\cos\theta_{r} , R_{\bot} , R_{\parallel}\sin\theta_{r}) \cos(\omega(t - \frac{\mathbf{r} \cdot \mathbf{s^{(r)}}}{v_{1}})) = \mathbf{E}_{0}^{(\mathbf{r})} \cos(\omega(t - \frac{\mathbf{r} \cdot \mathbf{s^{(r)}}}{v_{1}})),
  \end{split}
  \right.
\end{equation*}
and
\begin{equation*}
  \begin{cases}
    \mathbf{B}^{(\mathbf{i})}(\mathbf{r},t) & = \sqrt{\dfrac{\epsilon_{1}}{\mu_{1}}}(-A_{\bot}\cos\theta_{i} , -A_{\parallel} , A_{\bot}\sin\theta_{i}) \cos(\omega(t - \dfrac{\mathbf{r} \cdot \mathbf{s^{(i)}}}{v_{1}})) \\
    & =  \sqrt{\dfrac{\epsilon_{1}}{\mu_{1}}} \mathbf{B}_{0}^{(\mathbf{i})} \cos(\omega(t - \dfrac{\mathbf{r} \cdot \mathbf{s^{(i)}}}{v_{1}})), \\
    \mathbf{B}^{(\mathbf{t})}(\mathbf{r},t) & = \sqrt{\dfrac{\epsilon_{2}}{\mu_{2}}}(-A_{\bot}\cos\theta_{t} , -A_{\parallel} , A_{\bot}\sin\theta_{t}) \cos(\omega(t - \dfrac{\mathbf{r} \cdot \mathbf{s^{(t)}}}{v_{2}})) \\
    & =  \sqrt{\dfrac{\epsilon_{2}}{\mu_{2}}} \mathbf{B}_{0}^{(\mathbf{t})} \cos(\omega(t - \dfrac{\mathbf{r} \cdot \mathbf{s^{(t)}}}{v_{2}})), \\
    \mathbf{B}^{(\mathbf{r})}(\mathbf{r},t) & = \sqrt{\dfrac{\epsilon_{1}}{\mu_{1}}}(-A_{\bot}\cos\theta_{r} , -A_{\parallel} , A_{\bot}\sin\theta_{r}) \cos(\omega(t - \dfrac{\mathbf{r} \cdot \mathbf{s^{(r)}}}{v_{1}}))  \\
    & =  \sqrt{\dfrac{\epsilon_{1}}{\mu_{1}}} \mathbf{B}_{0}^{(\mathbf{r})} \cos(\omega(t - \dfrac{\mathbf{r} \cdot \mathbf{s^{(r)}}}{v_{1}})),
  \end{cases}
\end{equation*}
where $v_{1} = \displaystyle \frac{c}{\sqrt{\epsilon_{1} \mu_{1}}}$, $v_{2} = \displaystyle \frac{c}{\sqrt{\epsilon_{2} \mu_{2}}}$, $A$, $R$ and $T$ are the amplitude vectors and $\mathbf{s}^{(\mathbf{i})}$, $\mathbf{s}^{(\mathbf{t})}$ and $\mathbf{s}^{(\mathbf{r})}$ are the directions of propagation of the corresponding fields. The boundary conditions expressing the continuity of the tangential components of the electric and magnetic fields across the interface \cite{BW13}, then we have
\begin{equation}\label{2.12}
  \left\{
  \begin{split}
       & \mathbf{k} \times \mathbf{E}^{(\mathbf{i})}_{0} + \mathbf{k} \times \mathbf{E}^{(\mathbf{r})}_{0} =  \mathbf{k} \times \mathbf{E}^{(\mathbf{t})}_{0},  \\
       & \mathbf{k} \times \mathbf{B}^{(\mathbf{i})}_{0} + \mathbf{k} \times \mathbf{B}^{(\mathbf{r})}_{0} = \mathbf{k} \times \mathbf{B}^{(\mathbf{t})}_{0}.
  \end{split}
  \right.
\end{equation}
From (\ref{2.11}), we obtain
\begin{equation}\label{2.13}
  \left\{
  \begin{split}
       & A_{\bot} + R_{\bot} = T_{\bot}, \\
       & \cos\theta_{i}(A_{\parallel} - R_{\parallel}) = \cos\theta_{t}T_{\parallel}, \\
       & \frac{A_{\parallel}}{\sqrt{\dfrac{\epsilon_{1}}{\mu_{1}}}} +  \frac{R_{\parallel}}{\sqrt{\dfrac{\epsilon_{1}}{\mu_{1}}}} = \frac{T_{\parallel}}{\sqrt{\dfrac{\epsilon_{2}}{\mu_{2}}}}, \\
       & \cos\theta_{i}\left(\frac{A_{\bot}}{\sqrt{\dfrac{\epsilon_{1}}{\mu_{1}}}} - \frac{R_{\bot}}{\sqrt{\dfrac{\epsilon_{1}}{\mu_{1}}}}\right) = \cos\theta_{t} \frac{T_{\bot}}{\sqrt{\dfrac{\epsilon_{2}}{\mu_{2}}}}.
  \end{split}
  \right.
\end{equation}
Define the wave impedance of the medium as $z = \sqrt{\dfrac{\mu}{\epsilon}}$, then we obtain the following Fresnel formula from (\ref{2.13}):
\begin{equation}\label{2.14}
  \left\{
  \begin{split}
       & T_{\parallel} = \frac{2z_{1}\cos\theta_{i}}{z_{2} \cos\theta_{i} + z_{1}\cos\theta_{t}}A_{\parallel}, \\
       & T_{\bot} = \frac{2z_{1}\cos\theta_{i}}{z_{1}\cos\theta_{i} +  z_{2} \cos\theta_{t}}A_{\bot}, \\
       & R_{\parallel} =  \frac{z_{2} \cos\theta_{i} - z_{1}\cos\theta_{t}}{z_{2} \cos\theta_{i} + z_{1}\cos\theta_{t}}A_{\parallel},\\
       & R_{\bot} = \frac{z_{1}\cos\theta_{i} -  z_{2} \cos\theta_{t}}{z_{1}\cos\theta_{i} +  z_{2} \cos\theta_{t}}A_{\bot}.
  \end{split}
  \right.
\end{equation}
From Snell law (\ref{2.3}) and the fact that $x \cdot \nu = \cos\theta_{i}$, $m \cdot \nu = \cos\theta_{t}$, (\ref{2.14}) can be written as
\begin{equation}\label{2.15}
  \left\{
  \begin{split}
       & T_{\parallel} = \frac{2 z_{1} x\cdot (x-\kappa m)}{(z_{2} x + z_{1} m)\cdot (x - \kappa m)}A_{\parallel}, \\
       & T_{\bot} = \frac{2 z_{1} x\cdot (x-\kappa m)}{(z_{1} x + z_{2} m) \cdot (x - \kappa m)}A_{\bot}, \\
       & R_{\parallel} = \frac{(z_{2} x - z_{1} m)\cdot (x - \kappa m)}{(z_{2} x + z_{1} m)\cdot (x - \kappa m)}A_{\parallel}, \\
       & R_{\bot} = \frac{(z_{1} x - z_{2} m) \cdot (x - \kappa m)}{(z_{1} x + z_{2} m) \cdot (x - \kappa m)}A_{\bot}.
  \end{split}
  \right.
\end{equation}

Using Poynting vector (\ref{2.10}), the amount of energies of incident, transmitted and reflected waves leaving a unit area of the boundary per second is given by
\begin{equation*}
  \left\{
  \begin{split}
       & J^{(i)} = \vert \mathbf{S}^{\mathbf{i}} \vert \cos\theta_{i} = \frac{c}{4\pi} \sqrt{\frac{\epsilon_{1}}{\mu_{1}}} \vert \mathbf{E}_{0}^{(\mathbf{i})}\vert^{2} x \cdot \nu, \\
       & J^{(t)} = \vert \mathbf{S}^{\mathbf{t}} \vert \cos\theta_{t} = \frac{c}{4\pi} \sqrt{\frac{\epsilon_{2}}{\mu_{2}}} \vert \mathbf{E}_{0}^{(\mathbf{t})}\vert^{2} m \cdot \nu, \\
       & J^{(r)} = \vert \mathbf{S}^{\mathbf{r}} \vert \cos\theta_{r} = \frac{c}{4\pi} \sqrt{\frac{\epsilon_{1}}{\mu_{1}}} \vert \mathbf{E}_{0}^{(\mathbf{r})}\vert^{2} x \cdot \nu.
  \end{split}
  \right.
\end{equation*}
Then we can define the reflection and transmission coefficients as
\begin{equation*}
  \left\{
  \begin{split}
       & r_{\Gamma}(x) = \frac{J^{(r)}}{J^{(i)}} = \left(\frac{\vert \mathbf{E}_{0}^{(\mathbf{r})} \vert}{\vert \mathbf{E}_{0}^{(\mathbf{i})} \vert}\right)^{2}, \\
       & t_{\Gamma}(x) = \frac{J^{(t)}}{J^{(i)}} = \sqrt{\frac{\epsilon_{2}\mu_{1}}{\epsilon_{1}\mu_{2}}} \frac{m \cdot \nu}{x \cdot \nu}\left(\frac{\vert \mathbf{E}_{0}^{(\mathbf{t})} \vert}{\vert \mathbf{E}_{0}^{(\mathbf{i})}\vert}\right)^{2}.
  \end{split}
  \right.
\end{equation*}
From Fresnel formula (\ref{2.15}), we have
\begin{equation}\label{2.16}
  \begin{aligned}
  r_{\Gamma}(x) & = \left[ \frac{z_{2} + \kappa z_{1} - (z_{1} + \kappa z_{2}) x \cdot m}{z_{2} - \kappa z_{1} + (z_{1} - \kappa z_{2}) x \cdot m} \right]^{2} \frac{A_{\parallel}^{2}}{A_{\parallel}^{2} + A_{\bot}^{2}} \\
  & \quad + \left[ \frac{z_{1} + \kappa z_{2} - (z_{2} + \kappa z_{1}) x \cdot m}{z_{1} - \kappa z_{2} + (z_{2} - \kappa z_{1}) x \cdot m} \right]^{2} \frac{A_{\bot}^{2}}{A_{\parallel}^{2} + A_{\bot}^{2}},
  \end{aligned}
\end{equation}
and by conservation of energy, we have
\begin{equation}\label{2.17}
  t_{\Gamma}(x) = 1 - r_{\Gamma}(x).
\end{equation}

\begin{remark}
  Equations (\ref{2.16}) and (\ref{2.17}) are called Fresnel's equation and $r_{\Gamma}(x)$ and $t_{\Gamma}(x)$ are called Fresnel coefficients.
  \label{rem2.2}
\end{remark}

\begin{remark}
    From Snell law (\ref{2.3}) and Eqs (\ref{2.16}) and (\ref{2.17}), $r_{\Gamma}(x)$ and $t_{\Gamma}(x)$ are functions only depending on $x$ and $\nu$.
    \label{rem2.3}
\end{remark}

\section{Far field refraction problem for the case $\kappa <-1$ with loss of energy}\label{Section 3}

\sloppy{}

In this section, we study the far field refraction problem for $\kappa = \dfrac{n_{2}}{n_{1}}<-1$. We first give the definition and some properties of the refractor, then discuss properties of Fresnel coefficients and define the weak solution of the far field refraction problem for the case $\kappa <-1$ with loss of energy. Finally, the existence of weak solution in both discrete and general situations are proved. Recall from (\ref{2.6}) in Remark \ref{rem2.1}, we must have
\begin{equation*}
  \inf_{x \in \bar{\Omega}, m \in \bar{\Omega}^{\ast}} x \cdot m \geq \frac{1}{\kappa} + \varepsilon
\end{equation*}
for some $\varepsilon>0$. Hence we have
\begin{equation}\label{3.1}
  1 - \kappa x \cdot m \geq -\varepsilon \kappa.
\end{equation}

\subsection{Refractor and its properties}\label{sec3.1}
The definition of refractor in the case $\kappa<-1$ stems from \cite{St17}.

\begin{definition}
  A parameterized surface $\Gamma$ in $\mathbb{R}^{n}$ given by $\Gamma = \{\rho(x)x; \rho \in C(\bar{\Omega})\}$ is a refractor from $\bar{\Omega}$ to $\bar{\Omega}^{\ast}$ in the case $\kappa<-1$, if for any $x_{0} \in \bar{\Omega}$, there exists a semi-hyperboloid defined as $H(m,b) = \{\rho(x)x; ~\rho(x) = \dfrac{b}{1 - \kappa m \cdot x}, ~x\in S^{n-1}, ~x \cdot m \geq \dfrac{1}{\kappa}\}$, such that $\rho(x_{0}) = \dfrac{b}{1 - \kappa m\cdot x_{0}}$ and $\rho(x) \geq \dfrac{b}{1 - \kappa m\cdot x}$ for all $x\in \bar{\Omega}$. Such $H(m,b)$ is called a supporting hyperboloid to $\Gamma$ at the point $\rho(x_{0})x_{0}$.
  \label{def3.1}
\end{definition}

The following Figure \ref{fig3} shows the semi-hyperboloid which refracts all ray emitted from the source $O$ to a specific direction for the case $\kappa<-1$.

\begin{figure}[h]
  \centering
  \includegraphics[width=8.3cm]{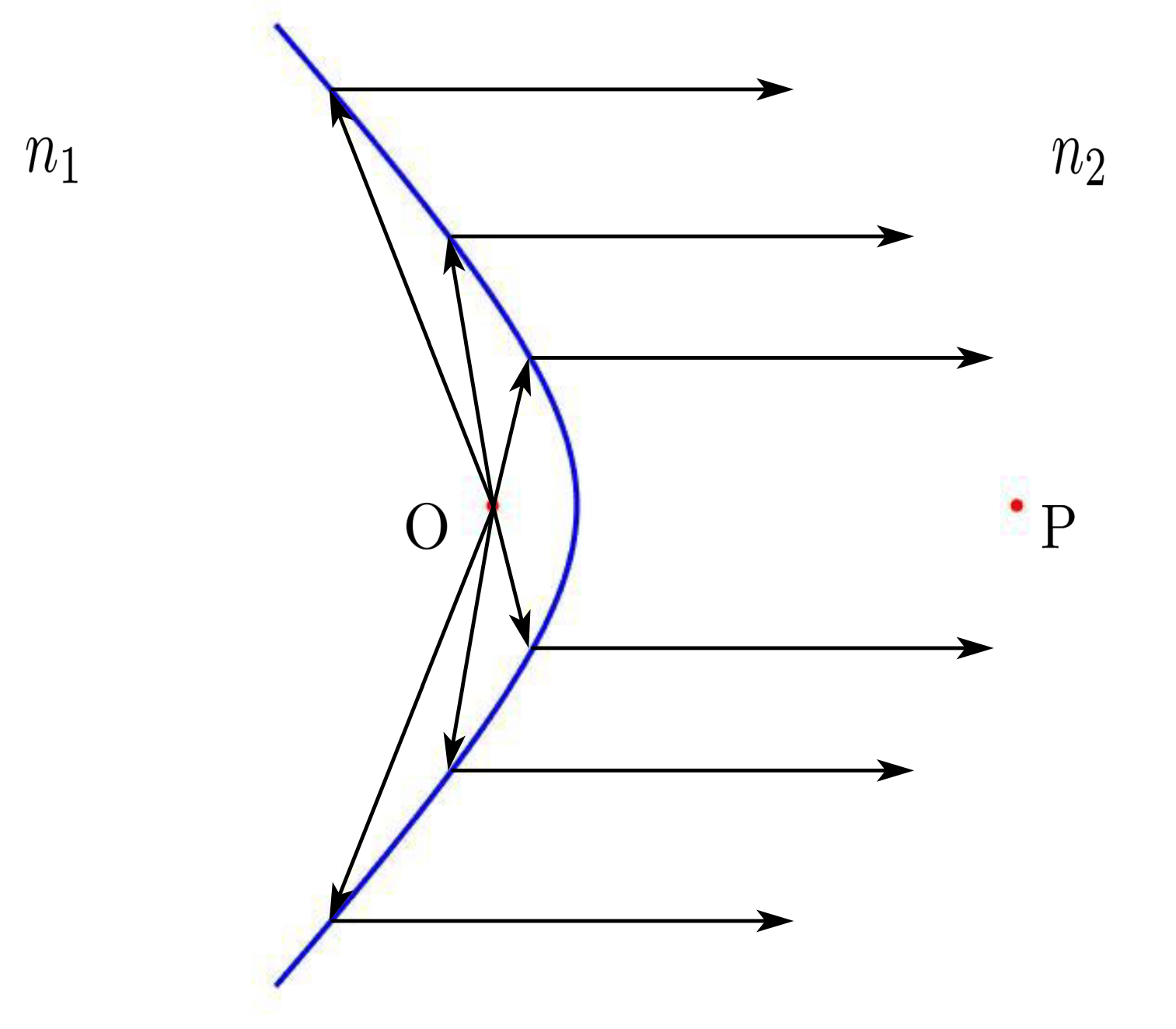}
\caption{Hyperboloid refracting when $\kappa<-1$, where $O$ and $P$ are focus of hyperboloid.}\label{fig3}
\end{figure}

Now we turn to discuss some properties of the refractor for the case $\kappa<-1$.

\begin{lemma}
  Any refractor is globally Lipschitz continuous on $\bar{\Omega}$, hence the set of singular points (set of discontinuous points) is a null set.
  \label{lem3.1}
\end{lemma}

\begin{proof}
  Suppose $\Gamma$ is a refractor from $\bar{\Omega}$ to $\bar{\Omega}^{\ast}$, parameterized by $\rho(x)x$, $x\in \bar{\Omega}$. Let $x\in \Omega$ and $H(m,b)$ supports $\Gamma$ at $\rho(x)x$. Then for any $y \in \bar{\Omega}$, we have $\rho(y) \geq \dfrac{b}{1 - \kappa m\cdot y}$ and $\rho(x) = \dfrac{b}{1 - \kappa m\cdot x}$.

  Since $\rho \in C(\bar{\Omega})$, then there exists $M>0$, such that $\rho(x) \leq M$ for all $x\in \bar{\Omega}$, thus $b \leq M$.
  Using (\ref{3.1}), we have
  \begin{align*}
    \vert \rho(x) - \rho(y) \vert & \leq \vert \frac{b}{1 - \kappa m\cdot x} - \frac{b}{1 - \kappa m\cdot y}  \vert \\
                      & = \vert \frac{b \kappa m \cdot (x-y)}{(1 - \kappa m \cdot x)(1 - \kappa m \cdot y)} \vert \\
                      & \leq -\frac{b\kappa \Vert x-y \Vert}{(-\varepsilon \kappa)^{2}} \\
                      & \leq -\frac{M}{\varepsilon^{2} \kappa} \Vert x-y \Vert.
  \end{align*}

Exchanging the roles of $x$ and $y$, then we can get $\vert \rho(x) - \rho(y)\vert \leq L \Vert x - y \Vert$ for some $L \geq 0$, hence $\rho$ is globally Lipschitz continuous on $\bar{\Omega}$. Then from Rademacher's theorem \cite{EV10}, we get the singular points set of $\rho$ is a null set.
\end{proof}

\begin{remark}
  If a refractor $\Gamma$ parameterized by $\rho$ has two distinct supporting semi-hyperboloid at $\rho(x)x$, then $\rho(x)x$ is a singular point of $\Gamma$.
  \label{rem3.1}
\end{remark}

\begin{lemma}
  Suppose $\Gamma = \{\rho(x)x ; ~\rho \in C(\bar{\Omega})\}$ is a refractor from $\bar{\Omega}$ to $\bar{\Omega}^{\ast}$, such that $\inf\limits_{x\in \bar{\Omega}}\rho(x) = 1$, then there exists a constant $C>0$ depending on $\varepsilon$ and $\kappa$, such that $\sup\limits_{x\in \bar{\Omega}}\rho(x) \leq C$.
  \label{lem3.2}
\end{lemma}

\begin{proof}
  Suppose that there exists $x_{0} \in \bar{\Omega}$, such that $\rho(x_{0}) = \sup\limits_{x\in \bar{\Omega}}\rho(x)$, and let $H(m_{0},b_{0})$ be the supporting hyperboloid to $\Gamma$ at $\rho(x_{0})x_{0}$. Then we have
    $\rho(x_{0}) = \dfrac{b_{0}}{1 - \kappa m_{0} \cdot x_{0}},$ and $\rho(x) \geq \dfrac{b_{0}}{1 - \kappa m_{0} \cdot x}$ for all $x \in \bar{\Omega}$. Since $\dfrac{b_{0}}{1 - \kappa} \leq \dfrac{b_{0}}{1 - \kappa m_{0} \cdot x}$ for all $x \in \bar{\Omega}$, so we have
    $$\frac{b_{0}}{1 - \kappa} \leq \inf_{x \in \bar{\Omega}}\frac{b_{0}}{1 - \kappa m_{0} \cdot x} \leq \inf_{x \in \bar{\Omega}} \rho(x) = 1.$$
    Hence we have $b_{0} \leq 1 - \kappa$.
    Consequently,
    $$\rho(x_{0}) = \frac{b_{0}}{1 - \kappa m_{0} \cdot x_{0}} \leq \frac{1 - \kappa}{1 - \kappa m_{0} \cdot x_{0}} \leq \frac{1-\kappa}{-\varepsilon \kappa}.$$
    Then we get $\sup\limits_{x\in \bar{\Omega}}\rho(x) \leq C$.
\end{proof}

Next, we define refractor mapping and trace mapping, and discuss some properties of them.

\begin{definition}
  Suppose the refractor $\Gamma = \{\rho(x)x; ~\rho \in C(\bar{\Omega})\}$ is given, the refractor mapping of $\Gamma$ is a multi-value map defined by
  \begin{equation}\label{3.2}
    \mathcal{N}_{\Gamma}(x_{0}) = \{m \in \bar{\Omega}^{\ast};~H(m,b) ~supports~\Gamma ~at ~\rho(x_{0})x_{0} ~ for~some ~b>0 \}.
  \end{equation}
  Given $m \in \bar{\Omega}^{\ast}$, the trace mapping of $\Gamma$ is defined by
  \begin{equation}\label{3.3}
    \mathcal{T}_{\Gamma}(m_{0}) = \{x \in \bar{\Omega};~m_{0}\in \mathcal{N}_{\Gamma}(x_{0})\}.
  \end{equation}
  \label{def3.2}
\end{definition}

\begin{lemma}
  If $m \in \bar{\Omega}^{\ast}$, then $\mathcal{T}_{\Gamma}(m)$ is a closed set in $\bar{\Omega}$.
  \label{lem3.3}
\end{lemma}

\begin{proof}
  Let $x_{n} \in \mathcal{T}_{\Gamma}(m)$ and $x_{n} \rightarrow x_{0}$, we need to prove that $x_{0} \in \mathcal{T}_{\Gamma}(m)$.

  For $x_{n} \in \mathcal{T}_{\Gamma}(m)$, then there exists $b>0$, such that $H(m,b)$ supports $\Gamma$ at $\rho(x_{n})x_{n}$, thus we have $\rho(x_{n}) = \dfrac{b}{1 - \kappa m \cdot x_{n}}$. For $x_{0} \in \bar{\Omega}$, then from Lemma \ref{lem3.1}, we have $\rho(x_{n})x_{n} \rightarrow \rho(x_{0})x_{0}$, so $\rho(x_{0}) = \dfrac{b}{1 - \kappa m \cdot x_{0}}$. Moreover, for $x \in \bar{\Omega}$, we have $\rho(x) \geq \dfrac{b}{1 - \kappa m \cdot x_{n}}$, then $\rho(x) \geq \dfrac{b}{1 - \kappa m \cdot x_{0}}$. So $H(m,b)$ supports $\Gamma$ at $\rho(x_{0})x_{0}$, that is, $x_{0} \in \mathcal{T}_{\Gamma}(m)$.
\end{proof}

\begin{lemma}
  For any $F \in \bar{\Omega}^{\ast}$, we have
  
   (a) $[\mathcal{T}_{\Gamma}(F)]^{c} \subseteq \mathcal{T}_{\Gamma}(F^{c})$;
  
   (b) The set $\mathcal{M} = \{F \subseteq \bar{\Omega}^{\ast};~\mathcal{T}_{\Gamma}(F)~is~Lebesgue~measurable\}$ is a $\sigma$-algebra containing all Borel sets in $\bar{\Omega}^{\ast}$.
  \label{lem3.4}
\end{lemma}

\begin{proof}
  $(a)$ If $x \in [\mathcal{T}_{\Gamma}(F)]^{c}$, then $\mathcal{N}_{\Gamma}(x) \cap F = \emptyset$, then $\mathcal{N}_{\Gamma}(x) \cap F^{c} \neq \emptyset$. Then by the definition of trace mapping (\ref{3.3}), we have $x \in \mathcal{T}_{\Gamma}(F^{c})$.

  $(b)$ We first prove that the set $\mathcal{M}$ is a $\sigma$-algebra.

  Obviously, we have $\mathcal{T}_{\Gamma}(\emptyset) = \emptyset$, $\mathcal{T}_{\Gamma}(\bar{\Omega}^{\ast}) = \bar{\Omega}$.
  For $F_{i} \in \mathcal{M}$, we have
  $$\mathcal{T}_{\Gamma}(\bigcup\limits_{i=1}^{\infty}F_{i}) = \bigcup\limits_{m \in \bigcup\limits_{i=1}^{\infty}F_{i}}\{x \in \bar{\Omega};~m\in \mathcal{N}_{\Gamma}(x)\} = \bigcup\limits_{i=1}^{\infty}\bigcup\limits_{m \in F_{i}}\{x \in \bar{\Omega};~m\in \mathcal{N}_{\Gamma}(x)\} =\bigcup\limits_{i=1}^{\infty} \mathcal{T}_{\Gamma}(F_{i}),$$
  so $\mathcal{M}$ is closed under countable union. Also for $F \in \mathcal{M}$, using $(a)$, we can get
  $$\mathcal{T}_{\Gamma}(F^{c}) = (\mathcal{T}_{\Gamma}(F^{c}) \cap [\mathcal{T}_{\Gamma}(F)]^{c}) \cup (\mathcal{T}_{\Gamma}(F^{c}) \cap \mathcal{T}_{\Gamma}(F)) = [\mathcal{T}_{\Gamma}(F)]^{c} \cup [\mathcal{T}_{\Gamma}(F^{c}) \cap \mathcal{T}_{\Gamma}(F)].$$
  Since $\vert \mathcal{T}_{\Gamma}(F^{c}) \cap \mathcal{T}_{\Gamma}(F) \vert = 0$ and $\mathcal{T}_{\Gamma}(F)$ is measurable, then $\mathcal{T}_{\Gamma}(F^{c})$ is measurable, hence $\mathcal{M}$ is closed under taking complements.

  Next, we prove that the set $\mathcal{M}$ contains all Borel sets in $\bar{\Omega}^{\ast}$.

  Indeed, choose a closed set $K \subseteq \bar{\Omega}^{\ast}$, obviously $K$ is compact. Take $x_{i} \in \mathcal{T}_{\Gamma}(K)$, then there exist $m_{i} \in \mathcal{N}_{\Gamma}(x_{i}) \cap K$. Suppose that $H(m_{i},b_{i})$ supports $\Gamma$ at $\rho(x_{i})x_{i}$, then we have $\rho(x_{i}) = \dfrac{b_{i}}{1 - \kappa m_{i}\cdot x_{i}}$ and $\rho(x) \geq \dfrac{b_{i}}{1 - \kappa m_{i}\cdot x}$ for all $x\in \bar{\Omega}$. For $1 - \kappa m_{i}\cdot x >0$ and $1 - \kappa m_{i}\cdot x_{i} >0$, we have
  $$\rho(x)(1 - \kappa m_{i}\cdot x) \geq b_{i} \quad \text{and} \quad \rho(x_{i})(1 - \kappa m_{i}\cdot x_{i}) = b_{i} ~\text{for all} ~x\in \bar{\Omega}.$$
  Assume that there exist constants $0 < a_{1} \leq a_{2}$, such that $a_{1} \leq \rho(x) \leq a_{2}$ on $\bar{\Omega}$, then we have
  $$b_{i} \leq \rho(x)(1 - \kappa m_{i}\cdot x) \leq a_{2}(1 - \kappa m_{i}\cdot x) \leq a_{2}(1 - \kappa),$$
  and
  $$b_{i} = \rho(x_{i})(1 - \kappa m_{i}\cdot x_{i}) \geq a_{1}(1 - \kappa m_{i}\cdot x_{i}) \geq a_{1}(-\kappa \varepsilon).$$
  Thus we have
  $$-a_{1}\kappa \varepsilon \leq b_{i} \leq a_{2}(1 - \kappa),$$
  so $b_{i}$s are bounded. Assume that there exist subsequences $x_{i}\rightarrow x_{0} \in \bar{\Omega}$, $m_{i}\rightarrow m_{0} \in K$ and $b_{i} \rightarrow b_{0}$ as $i \rightarrow \infty$. Then for $x \in \bar{\Omega}$ and all $i$, we have

  $$\rho(x)(1 - \kappa m_{i} \cdot x) \geq b_{i} \quad \text{and} \quad \rho(x_{i})(1 - \kappa m_{i} \cdot x_{i}) = b_{i}.$$
  Taking the limit as $i \rightarrow \infty$, we have

  $$\rho(x)(1 - \kappa m_{0} \cdot x) \geq b_{0} \quad \text{and} \quad \rho(x_{0})(1 - \kappa m_{0} \cdot x_{0}) = b_{0}.$$
  Hence $H(m_{0},b_{0})$ supports $\Gamma$ at $\rho(x_{0})x_{0}$ and $x_{0} \in \mathcal{T}_{\Gamma}(m_{0})$. Consequently, $\mathcal{T}_{\Gamma}(K)$ is compact, hence $\mathcal{M}$ contains all Borel sets in $\bar{\Omega}^{\ast}$.
\end{proof}

\begin{lemma}
  Suppose that $H(m_{k},b_{k})$ is a sequence of semi-hyperboloid, and $m_{k} \rightarrow m_{0}$, $b_{k} \rightarrow b_{0}$ as $k\rightarrow \infty$. Let $z_{k} \in H(m_{k},b_{k})$ with $z_{k} \rightarrow z_{0}$ as $k\rightarrow \infty$. Then $z_{0} \in H(m_{0},b_{0})$, and the normal $\nu_{k}(z_{k})$ to the semi-hyperboloid $H(m_{k},b_{k})$ at $z_{k}$ satisfies $\nu_{k}(z_{k}) \rightarrow \nu(z_{0})$ the normal to the semi-hyperboloid $H(m_{0},b_{0})$ at the point $z_{0}$.
  \label{lem3.5}
\end{lemma}

\begin{proof}
  The Cartesian coordinate of the equation of $H(m_{k},b_{k})$ is $\vert z \vert - \kappa m_{k}z = b_{k}$, then the normal vector at $z$ is $\nu_{k}(z) = \dfrac{z}{\vert z \vert} - \kappa m_{k}$, so we have
  $$\nu_{k}(z_{k}) = \frac{z_{k}}{\vert z_{k} \vert} - \kappa m_{k} \rightarrow \frac{z_{0}}{\vert z_{0} \vert} - \kappa m = \nu(z_{0}).$$
\end{proof}

\begin{lemma}
  Assume that $\Gamma_{k} = \{\rho_{k}(x)x;~x\in \bar{\Omega}\}$, $k \geq 1$ is a sequence of refractors from $\bar{\Omega}$ to $\bar{\Omega}^{\ast}$. Suppose that $0<a_{1}\leq \rho_{k} \leq a_{2}$ and $\rho_{k} \rightarrow \rho$ uniformly on $\bar{\Omega}$. Then we have
  
   (a) $\Gamma := \{\rho(x)x;~x\in \bar{\Omega}\}$ is a refractor from $\bar{\Omega}$ to $\bar{\Omega}^{\ast}$;
  
   (b) For any compact set $K\subseteq \bar{\Omega}^{\ast}$,
  $$\varlimsup_{k\rightarrow \infty} \mathcal{T}_{\Gamma_{k}}(K) \subseteq \mathcal{T}_{\Gamma}(K);$$
  
   (c) For any open set $G \subseteq \bar{\Omega}^{\ast}$,
  $$\mathcal{T}_{\Gamma}(G) \subseteq \varliminf_{k\rightarrow \infty}\mathcal{T}_{\Gamma_{k}}(G) \cup E,$$
  where $E$ is the singular set of $\Gamma$.
  \label{lem3.6}
\end{lemma}

\begin{proof}
  $(a)$ Obviously we have $\rho \in C(\bar{\Omega})$ and $\rho >0$. Fix $x_{0} \in \bar{\Omega}$, then there exist $m_{k} \in \bar{\Omega}^{\ast}$ and $b_{k} >0$, such that $H(m_{k},b_{k})$ supports $\Gamma_{k}$ at $\rho(x_{0})x_{0}$, thus
  $$\rho_{k}(x) \geq \frac{b_{k}}{1 - \kappa m_{k}\cdot x} \quad \text{and} \quad \rho_{k}(x_{0}) = \frac{b_{k}}{1 - \kappa m_{k}\cdot x_{0}}~\text{for all} ~x\in \bar{\Omega}.$$
  So for all $x\in \bar{\Omega}$ and $k$, we have
  $$\frac{b_{k}}{1 - \kappa m_{k}\cdot x_{0}} \geq a_{1} \quad \text{and} \quad \frac{b_{k}}{1 - \kappa m_{k}\cdot x}\leq a_{2},$$
  hence
  $$a_{1}(1 - \kappa m_{k}\cdot x_{0}) \leq b_{k} \leq a_{2}(1 - \kappa m_{k}\cdot x).$$
  Combing with (\ref{3.1}), we have
  $$-a_{1}\varepsilon \kappa \leq b_{k} \leq a_{2}(1 - \kappa)$$
  for all $k$. Then there exist $m_{0} \in \bar{\Omega}^{\ast}$ and $b_{0} >0$, such that $m_{k}\rightarrow m_{0}$ and $b_{k}\rightarrow b_{0}$. Hence we have
  $$\rho(x_{0}) = \lim_{k\rightarrow \infty}\rho_{k}(x_{0}) = \lim_{k\rightarrow \infty}\frac{b_{k}}{1 - \kappa m_{k}\cdot x_{0}} = \frac{b_{0}}{1 - \kappa m_{0}\cdot x_{0}}$$
  and
  $$\rho(x) = \lim_{k\rightarrow \infty}\rho_{k}(x) \geq \lim_{k\rightarrow \infty}\frac{b_{k}}{1 - \kappa m_{k}\cdot x} = \frac{b_{0}}{1 - \kappa m_{0}\cdot x}.$$
  for all $x\in \bar{\Omega}$, hence $H(m_{0},b_{0})$ supports $\Gamma$ at $\rho(x_{0})x_{0}$. So $\Gamma$ is a refractor.

   $(b)$ Let $x_{0} \in \varlimsup\limits_{k\rightarrow \infty} \mathcal{T}_{\Gamma_{k}}(K)$. Without loss of generality, we assume that $x_{0} \in \mathcal{T}_{\Gamma_{k}}(K)$ for all $k \geq 1$, then there exist $m_{k} \in \mathcal{N}_{\Gamma_{k}(x_{0})} \cap K$ and $b_{k} >0$, such that
   $$\rho_{k}(x_{0}) = \frac{b_{k}}{1 - \kappa m_{k}\cdot x_{0}} \quad \text{and} \quad \rho_{k}(x) \geq \frac{b_{k}}{1 - \kappa m_{k}\cdot x}$$
   for all $x \in \bar{\Omega}$. We may assume that there exist $m_{0} \in K$ and $b_{0} >0$, such that $m_{k}\rightarrow m_{0}$ and $b_{k}\rightarrow b_{0}$, then as in proof of $(a)$, $H(m_{0},b_{0})$ supports $\Gamma$ at $\rho(x_{0})x_{0}$, hence $x_{0} \in \mathcal{T}_{\Gamma}(m_{0})$. Consequently, $x \in \mathcal{T}_{\Gamma}(K)$.

   $(c)$ Suppose that $G \subseteq \bar{\Omega}^{\ast}$ is a open set, then $G^{c}$ is a compact set. From $(b)$, we have
   $$\varlimsup_{k\rightarrow \infty}\mathcal{T}_{\Gamma_{k}}(G^{c}) \subseteq \mathcal{T}_{\Gamma}(G^{c}).$$
   Besides, by Lemma \ref{lem3.4}, we also have
   $$\varlimsup_{k\rightarrow \infty}[\mathcal{T}_{\Gamma_{k}}(G)]^{c} \subseteq \varlimsup_{k\rightarrow \infty}[\mathcal{T}_{\Gamma_{k}}(G)]^{c} \cup [\mathcal{T}_{\Gamma_{k}}(G) \cap \mathcal{T}_{\Gamma_{k}}(G^{c})] = \varlimsup_{k\rightarrow \infty}\mathcal{T}_{\Gamma_{k}}(G^{c}).$$
   From $(b)$, we have
   \begin{equation}\label{3.4}
     \varlimsup_{k\rightarrow \infty}[\mathcal{T}_{\Gamma_{k}}(G)]^{c} \subseteq \mathcal{T}_{\Gamma}(G^{c}) = [\mathcal{T}_{\Gamma}(G)]^{c} \cup [\mathcal{T}_{\Gamma}(G) \cap \mathcal{T}_{\Gamma}(G^{c})].
   \end{equation}
   Taking complements in (\ref{3.4}), we have
   $$\varliminf_{k\rightarrow \infty}\mathcal{T}_{\Gamma_{k}}(G) \supseteq \mathcal{T}_{\Gamma}(G) \cap [\mathcal{T}_{\Gamma}(G) \cap \mathcal{T}_{\Gamma}(G^{c})].$$
   Hence
   $$\mathcal{T}_{\Gamma}(G) \cap [\varliminf_{k\rightarrow \infty}\mathcal{T}_{\Gamma_{k}}(G)]^{c} \cup E \subseteq \varliminf_{k\rightarrow \infty}\mathcal{T}_{\Gamma_{k}}(G) \cup E.$$
   However, $\varliminf\limits_{k\rightarrow \infty}\mathcal{T}_{\Gamma_{k}}(G) \subseteq E$, hence
   $$\mathcal{T}_{\Gamma}(G) \subseteq \mathcal{T}_{\Gamma}(G) \cup E \subseteq \varliminf_{k\rightarrow \infty}\mathcal{T}_{\Gamma_{k}}(G) \cup E.$$

\end{proof}

\subsection{Properties of Fresnel coefficients}

Recall the Fresnel coefficients in (\ref{2.16}) and (\ref{2.17})
\begin{equation*}
  \begin{cases}
    r_{\Gamma}(x) \!\!\!\!&= \left[\dfrac{z_{2} + \kappa z_{1} - (z_{1} + \kappa z_{2}) x \cdot m}{z_{2} - \kappa z_{1} + (z_{1} - \kappa z_{2}) x \cdot m} \right]^{2} \dfrac{A_{\parallel}^{2}}{A_{\parallel}^{2} + A_{\bot}^{2}} \\
    \!\!\!\!&\quad+ \left[ \dfrac{z_{1} + \kappa z_{2} - (z_{2} + \kappa z_{1}) x \cdot m}{z_{1} - \kappa z_{2} + (z_{2} - \kappa z_{1}) x \cdot m}\right]^{2} \dfrac{A_{\bot}^{2}}{A_{\parallel}^{2} + A_{\bot}^{2}},  \\
    t_{\Gamma}(x) \!\!\!\!&= 1 - r_{\Gamma}(x).
  \end{cases}
\end{equation*}
We first discuss the boundedness of $r_{\Gamma}(x)$ and $t_{\Gamma}(x)$.

For simplicity, let $\sigma = \displaystyle \frac{z_{2}}{z_{1}} = \displaystyle \sqrt{\frac{\mu_{2}\epsilon_{1}}{\mu_{1}\epsilon_{2}}} > 0$ and introduce a function
\begin{equation}\label{3.5}
  \psi(t) := \left[\frac{\sigma + \kappa -(1 + \kappa \sigma)t}{\sigma - \kappa + (1 - \kappa \sigma)t}\right]^{2} \alpha + \left[\frac{1 + \kappa \sigma - (\sigma + \kappa)t}{1 - \kappa \sigma + (\sigma - \kappa)t}\right]^{2} \beta,
\end{equation}
where $\alpha = \dfrac{A_{\parallel}^{2}}{A_{\parallel}^{2} + A_{\bot}^{2}}$, $\beta = \dfrac{A_{\bot}^{2}}{A_{\parallel}^{2} + A_{\bot}^{2}}$. Then $r_{\Gamma}(x) = \psi(x \cdot m)$, $t_{\Gamma}(x) = 1- \psi(x \cdot m)$. From (\ref{2.6}), we know $t\in\left[\dfrac{1}{\kappa} + \varepsilon,1\right]$. We denote
$$p(t) = \frac{\sigma + \kappa -(1 + \kappa \sigma)t}{\sigma - \kappa + (1 - \kappa \sigma)t} \quad \text{and} \quad q(t) = \frac{1 + \kappa \sigma - (\sigma + \kappa)t}{1 - \kappa \sigma + (\sigma - \kappa)t}.$$

For $p(t)$, we have $p'(t) = \displaystyle \frac{2 \sigma (\kappa^{2} - 1)}{[\sigma - \kappa + (1 - \kappa \sigma)t]^{2}}$.  For $\kappa <-1$, then $\kappa^{2} - 1 >0$, so $p(t)$ increases on $\left[\dfrac{1}{\kappa} + \varepsilon,1\right]$. Hence 
$$p^{2}(t)_{\max} = \max\left\{p^{2}\left(\dfrac{1}{\kappa} + \varepsilon\right), p^{2}(1)\right\}.$$
We have $p^{2}(1) = \left[\dfrac{\sigma - 1}{\sigma + 1}\right]^{2}$, $p^{2}\left(\dfrac{1}{\kappa} + \varepsilon\right) = \left[\dfrac{\kappa^{2} - 1 - \varepsilon \kappa(1 + \kappa \sigma)}{1 - \kappa^{2} + \varepsilon \kappa(1 - \kappa \sigma)}\right]^{2}$. For $\varepsilon$ is small enough, then $p^{2}(t)_{\max} = p^{2}\left(\dfrac{1}{\kappa} + \varepsilon\right)$.

For $q(t)$, we have $q'(t) = \dfrac{2 \sigma (\kappa^{2} - 1)}{[1 - \kappa \sigma + (\sigma - \kappa)t]^{2}}$. For $\kappa <-1$, then $\kappa^{2} - 1 >0$, then $q(t)$ increases on $\left[\dfrac{1}{\kappa} + \varepsilon,1\right]$. Hence 
$$q^{2}(t)_{\max} = \max\left\{q^{2}\left(\dfrac{1}{\kappa} + \varepsilon\right), q^{2}(1)\right\}.$$
We have $q^{2}(1) = \left[-\dfrac{\sigma - 1}{\sigma + 1}\right]^{2}$, $q^{2}\left(\dfrac{1}{\kappa} + \varepsilon\right) = \left[\dfrac{\sigma (\kappa^{2} - 1) - \varepsilon \kappa^{2}(\sigma + 1)}{\sigma (1 - \kappa^{2}) + \varepsilon \kappa (\sigma - \kappa)}\right]^{2}$. For $\varepsilon$ is small enough, then $q^{2}(t)_{\max} = q^{2}\left(\dfrac{1}{\kappa} + \varepsilon\right)$.

From above analysis, we obtain the following proposition:
\begin{proposition}
  Suppose that $\bar{\Omega}$ and $\bar{\Omega}^{\ast}$ satisfy (\ref{2.6}), $\Gamma$ is a refractor from $\bar{\Omega}$ to $\bar{\Omega}^{\ast}$, then there exists a constant $C_{\varepsilon}$ associated with $\varepsilon$, such that $r_{\Gamma}(x) \leq C_{\varepsilon}$ and $C_{\varepsilon} < t_{\Gamma}(x) <1$.
  \label{prop3.1}
\end{proposition}

Next, we discuss the continuity of $t_{\Gamma}(x)$.
\begin{proposition}
  Suppose that $\Gamma = \{\rho(x)x;~x\in \bar{\Omega}\}$ is refractor from $\bar{\Omega}$ to $\bar{\Omega}^{\ast}$ and $E$ is the singular set of $\Gamma$, then $t_{\Gamma}(x)$ is continuous on $\bar{\Omega}\backslash E$.
  \label{prop3.2}
\end{proposition}

\begin{proof}
  To prove $t_{\Gamma}(x)$ is continuous on $\bar{\Omega}\backslash E$, we only need to prove $r_{\Gamma}(x)$ is continuous on $\bar{\Omega}\backslash E$. From previous analysis, we can assume that there exist constants $C_{1},C_{2}>0$, such that $C_{1} \leq \rho(x) \leq C_{2}$. From (\ref{2.16}), we know that $r_{\Gamma}(x)$ is a function $\phi(x) = G(x,\nu(x))$ define on $\bar{\Omega}\backslash E$, and $G(x,m)$ is continuous on $\bar{\Omega} \times \bar{\Omega}^{\ast}$.

  To prove $r_{\Gamma}(x)$ is continuous on $\bar{\Omega}\backslash E$, we only need to prove $r_{\Gamma}(x)$ is both upper and lower semi-continuous on $\bar{\Omega}\backslash E$. We first prove  $r_{\Gamma}(x)$ is upper semi-continuous on $\bar{\Omega}\backslash E$, that is, for any $\alpha \in \mathbb{R}$, the set $M_{\alpha} = \{x \in \bar{\Omega}\backslash E;~\phi(x)\leq \alpha\}$ is a closed set in $\bar{\Omega}\backslash E$. Then we need to prove that for a sequence $x_{k} \in M_{\alpha}$ and $x_{0} \in \bar{\Omega}\backslash E$, if $x_{k} \rightarrow x_{0}$, then $x_{0} \in M_{\alpha}$.

  We claim that for $x_{k},x_{0} \in \bar{\Omega}\backslash E$, if $x_{k} \rightarrow x_{0}$, then there exists a subsequence $x_{k_{j}}$, such that $\nu(x_{k_{j}})\rightarrow \nu(x_{0})$ as $j\rightarrow \infty$.

  Indeed, suppose that $H(m_{k},b_{k})$ support $\Gamma$ at $\rho(x_{k})x_{k}$, then we have
  $$\rho(x) \geq \frac{b}{1 - \kappa m_{k} \cdot x} \quad \text{and} \quad \rho(x_{k}) = \frac{b_{k}}{1 - \kappa m_{k} \cdot x_{k}}.$$
  Hence from (\ref{3.1}), we have
  $$-C_{1}\kappa \varepsilon \leq b_{k} \leq C_{2}(1 - \kappa).$$
  Then there exist subsequence $b_{k_{j}}\rightarrow b_{0}$ and $m_{k_{j}}\rightarrow m_{0}$ as $j \rightarrow \infty$. From Lemma \ref{lem3.5}, the claim holds true.

  Consequently, if $x_{k} \in M_{\alpha}$, then $\phi(x_{k}) = G(x_{k},\nu(x_{k})) \leq \alpha$. However, from claim, there exists a subsequence $x_{k_{j}}$, such that $\nu(x_{k_{j}})\rightarrow \nu(x_{0})$ as $j\rightarrow \infty$. Then for $G$ is continuous, we know that $r_{\Gamma}(x)$ is upper semi-continuous on $\bar{\Omega}\backslash E$.

  Using the similar argument, we can prove that $r_{\Gamma}(x)$ is lower semi-continuous on $\bar{\Omega}\backslash E$. Then $r_{\Gamma}(x)$ is continuous on $\bar{\Omega}\backslash E$.
\end{proof}

\begin{remark}
  From Lemma \ref{lem3.1}, the singular points set of $\rho$ is a null set, then $r_{\Gamma}(x)$ is well-defined on $\Omega$ a.e., hence $r_{\Gamma}(x)$ is measurable in $\Omega$.
  \label{rem3.2}
\end{remark}

From above analysis, we can get the following lemma and theorem, which are useful in proving the existence of the weak solution.

\begin{lemma}
  Let $\Gamma_{k}$ and $\Gamma$ be refractors with defining functions $\rho_{k}(x)$ and $\rho(x)$, the corresponding fresnel coefficients are $t_{k}$ and $t$. Suppose that $\rho_{k}\rightarrow \rho$ pointwise in $\bar{\Omega}$ and there exist constants $C_{1},C_{2}>0$, such that $C_{1} \leq \rho_{k}(x) \leq C_{2}$ in $\bar{\Omega}$. Then for $y \notin E$, there exists a subsequence $t_{k_{j}}(y)\rightarrow t(y)$ as $j\rightarrow \infty$, where $E$ is the union of singular points of refractors $\Gamma_{k}$ and $\Gamma$.
  \label{lem3.7}
\end{lemma}

\begin{proof}
  Given $y \notin E$ and $k$, there exist $b_{k}>0$ and $m_{k} \in \bar{\Omega}^{\ast}$, such that
  $$\rho_{k}(y) = \frac{b_{k}}{1 - \kappa m_{k} \cdot y} \quad \text{and} \quad \rho_{k}(z) \geq \frac{b_{k}}{1 - \kappa m_{k} \cdot z}~\text{for all}~ z\in \bar{\Omega}. $$
  So we have $C_{1} \leq \dfrac{b_{k}}{1 - \kappa m_{k} \cdot y} \leq C_{2}$, then from (\ref{3.1}), we get
  $$-C_{1}\kappa \varepsilon \leq b_{k} \leq C_{2}(1 - \kappa).$$
  So $b_{k}$s are bounded and away from 0 and $\infty$, then there exist subsequence $b_{k_{j}}\rightarrow b >0$ and $m_{k_{j}}\rightarrow m \in \bar{\Omega}^{\ast}$, hence $H(m,b)$ supports $\Gamma$ at $y\rho(y)$, so $y \in \mathcal{T}_{\Gamma}(m)$. For $y \notin E$, the normal $\nu_{k_{j}}(y)$ to the semi-hyperboloid  $H(m_{k_{j}},b_{k_{j}})$ equals to the normal to the refractor $\Gamma_{k_{j}}$ at $y$, and the normal $\nu(y)$ to the semi-hyperboloid $H(m,b)$ equals to the normal to the refractor $\Gamma$ at $y$. Since $H(m_{k_{j}},b_{k_{j}}) \rightarrow H(m,b)$ as $j\rightarrow \infty$, then $\nu_{k_{j}}(y) \rightarrow \nu(y)$ for $y \notin E$ as  $j\rightarrow \infty$. So we have $t_{k_{j}}(y)\rightarrow t(y)$ as $j \rightarrow \infty$.
\end{proof}

\begin{theorem}
  Assume that the hypotheses and notations of Lemma \ref{lem3.7} hold, and let $F \subseteq \bar{\Omega}^{\ast}$ be a compact set, set $F_{k} = \mathcal{T}_{\Gamma_{k}}(F)$. Then for all $y \notin E$, we have
  \begin{equation}\label{3.6}
    (a) \varlimsup_{k \rightarrow \infty} \chi_{F_{k}}(y)t_{k}(y) = t(y)\varlimsup_{k \rightarrow \infty} \chi_{F_{k}}(y);
  \end{equation}
  \begin{equation}\label{3.7}
    (b) \varliminf_{k \rightarrow \infty} \chi_{F_{k}}(y)t_{k}(y) = t(y)\varliminf_{k \rightarrow \infty} \chi_{F_{k}}(y).
  \end{equation}
  \label{thm3.1}
\end{theorem}

\begin{proof}
  See Theorem 5.5 in \cite{GM13}.
\end{proof}

\subsection{Weak solution of the far field refraction problem for the case $\kappa <-1$ with loss of energy}
In this subsection, we define the weak solution of the far field refraction problem for the case $\kappa <-1$ with loss of energy. We first give the definition of refractor measure originated from \cite{GM13}.

\begin{definition}
  Suppose $\Gamma$ is a refractor from $\bar{\Omega}$ to $\bar{\Omega}^{\ast}$, $f \in L^{1}(\bar{\Omega})$ and $\inf\limits_{\bar{\Omega}}f >0$. The refractor measure associated with $\Gamma$ and $f$ is defined by a set function on Borel subsets of $\bar{\Omega}^{\ast}$:
  \begin{equation}\label{3.8}
    G_{\Gamma}(F) := \int_{\mathcal{T}_{\Gamma}(F)}f(x)t_{\Gamma}(x) \, dx,
  \end{equation}
  where $dx$ is the surface measure on $S^{n-1}$.
  \label{def3.3}
\end{definition}

\begin{remark}
  $G_{\Gamma}(F)$ is a finite Borel measure defined on $\mathcal{M}$, where $\mathcal{M}$ is defined in Lemma \ref{lem3.4} $(b)$.
  \label{rem3.3}
\end{remark}

Now we can define the weak solution of the far field refraction problem for the case $\kappa <-1$ with loss of energy.
\begin{definition}
  Suppose that $\mu$ is a Radon measure on the Borel subset of $\bar{\Omega}^{\ast}$ and $f \in L^{1}(\bar{\Omega})$, a refractor $\Gamma$ is a weak solution of the far field refraction problem for the case $\kappa <-1$ with emitting illumination intensity $f(x)$ and prescribe refracted illumination intensity $\mu$ if for any Borel set $\omega \subseteq \bar{\Omega}^{\ast}$, there holds:
  \begin{equation}\label{3.9}
    G_{\Gamma}(\omega) = \int_{\mathcal{T}_{\Gamma}(\omega)}f(x)t_{\Gamma}(x)dx \geq \mu(\omega).
  \end{equation}
  \label{def3.4}
\end{definition}

\begin{remark}
  Since a small portion of energy is used for internal reflection, a little extra energy is required to ensure that light can be refracted into $\bar{\Omega}^{\ast}$, so we use ``$\geq$'' in (\ref{3.9}). 
\label{rem3.4}
\end{remark}

From Definition \ref{def3.4}, we can prove that the weak solution is unique up to a multiplicative constant.

\begin{theorem}
  If $\Gamma = \{\rho(x)x;~x\in \bar{\Omega}\}$ is a weak solution of the refraction problem, then for any $c>0$, $c\Gamma = \{c\rho(x)x;~x\in \bar{\Omega}\}$ is also a weak solution of the refraction problem.
  \label{thm3.2}
\end{theorem}

\begin{proof}
  If $H(m,b)$ supports $\Gamma$ at $\rho(x)x$, then $H(m,cb)$ supports $c\Gamma$ at $c\rho(x)x$. Then for any $\omega \in \bar{\Omega}^{\ast}$, we have $\mathcal{T}_{\Gamma}(\omega) = \mathcal{T}_{c\Gamma}(\omega)$ and $t_{\Gamma}(x) = t_{c\Gamma}(x)$, hence $c\Gamma$ is also a weak solution.
\end{proof}

The existence of weak solution is discussed in the following two subsections.

\subsection{Existence of the weak solution when $\mu$ is discrete measure}

In this subsection, we assume that $\mu$ equals finite sum of $\delta$-measures, hence all rays are refracted into finite directions. Based on this assumption, we establish the existence of the weak solution of the far field refraction problem for the case $\kappa <-1$ with loss of energy when $\mu$ is discrete measure.

\begin{remark}
  Suppose that $m_{1},m_{2},\ldots, m_{l},~l \geq 2$ are discrete points in $\bar{\Omega}^{\ast}$, then for $\mathbf{b} = (b_{1},b_{2},\cdots ,b_{l}) \in \mathbb{R}^{l},~b_{i}>0$, the refractor is defined as
  \begin{equation}\label{3.10}
    \Gamma(\mathbf{b}) = \{\rho(x)x;~x\in \bar{\Omega},~\rho(x)=\max_{1\leq i \leq l}\frac{b_{i}}{1 - \kappa m_{i}\cdot x}\}.
  \end{equation}
  \label{rem3.5}
\end{remark}

Now we show the existence of the weak solution when $\mu$ equals the linear combination of the $\delta$-measures at $m_{1},m_{2},\ldots, m_{l}$. 

\begin{theorem}
  Suppose that $f \in L^{1}(\bar{\Omega})$ and $\inf\limits_{x \in \bar{\Omega}}f(x)>0$, $m_{1},m_{2},\ldots, m_{l},~l \geq 2$ are discrete points in $\bar{\Omega}^{\ast}$, $g_{1}, g_{2}, \ldots, g_{l}>0$. Let $\mu$ be the Borel measure defined on $\bar{\Omega}^{\ast}$ by $\mu = \sum\limits_{i = 1}^{l}g_{i}\delta_{m_{i}}(\omega)$, where $\omega \in \bar{\Omega}^{\ast}$ is Borel set. Also suppose that $$\displaystyle\int_{\bar{\Omega}}f(x)dx \geq \dfrac{1}{1 - C_{\varepsilon}} \mu(\bar{\Omega}^{\ast}),$$ where $C_{\varepsilon}$ is defined in Proposition \ref{prop3.1}. Then there exist $\mathbf{b}_{0} \in \mathbb{R}^{l}$ and refractor $\Gamma(\mathbf{b}_{0})$, such that
  $$\int_{\mathcal{T}_{\Gamma{(\mathbf{b_{0}})}(m_{i})}}f(x)t_{\Gamma(\mathbf{b}_{0})}(x)dx = g_{i}$$
  for $i = 2,\ldots,l$, and
  $$\int_{\mathcal{T}_{\Gamma{(\mathbf{b_{0}})}(m_{1})}}f(x)t_{\Gamma(\mathbf{b}_{0})}(x)dx > g_{1},$$
  namely the weak solution of the far field refraction problem for the case $\kappa <-1$ with loss of energy exists.
  \label{thm3.3}
\end{theorem}

\begin{remark}
    If $l = 1$, this problem might be overdetermined. In this case, we have $\Gamma(\mathbf{b})$ equals $H(m_{1},b_{1})$, hence this predetermines the value of $t_{\Gamma}(x)$ and $\displaystyle\int_{\bar{\Omega}}f(x)t_{\Gamma}(x)dx$, so we must have
    $$\int_{\mathcal{T}_{\Gamma{(\mathbf{b_{0}})}(m_{1})}}f(x)t_{\Gamma(\mathbf{b}_{0})}(x)dx > g_{1}.$$
    \label{rem3.6}
\end{remark}

In order to prove Theorem \ref{thm3.3}, we need some lemmas.

\begin{lemma}
  Suppose that $f \in L^{1}(\bar{\Omega})$ and $\inf\limits_{x \in \bar{\Omega}}f(x)>0$, $m_{1},m_{2},\ldots, m_{l},~l \geq 2$ are discrete points in $\bar{\Omega}^{\ast}$, $g_{1}, g_{2}, \ldots, g_{l}>0$. Suppose that $W \subseteq \mathbb{R}^{l}$ is a set defined by $W := \{\mathbf{b} = (1,b_{2},\ldots, b_{l});~b_{i}>0 \ {for} \ i=2, \ldots, l\}$, and for any $\mathbf{b} \in W$, $\Gamma(\mathbf{b})$ satisfies $G_{\Gamma(\mathbf{b})}(m_{i}) = \displaystyle\int_{\mathcal{T}_{\Gamma(\mathbf{b})}(m_{i})}f(x)t_{\Gamma(\mathbf{b})}(x)dx \leq g_{i}$ for $i = 2,\ldots,l$. Then we have:
  
   (a) $W \neq \emptyset$;
  
   (b) If $\mathbf{b} = (1,b_{2}, \ldots, b_{l}) \in W$, then $b_{i}<\dfrac{1 - \kappa}{-\varepsilon \kappa}$ for $i=2, \ldots, l$.
  \label{lem3.8}
\end{lemma}

\begin{proof}
  $(a)$ If for some $i \neq 1$, $H(m_{i},b_{i})$ supports $\Gamma$ at $\rho(x)x$, then we have
  $$\frac{1}{1 - \kappa} \leq \frac{1}{1 - \kappa m_{1} \cdot x} \leq \rho(x) = \frac{b_{i}}{1 - \kappa m_{i} \cdot x} \leq \frac{b_{i}}{-\kappa \varepsilon},$$
  so $b_{i} \geq \dfrac{-\kappa \varepsilon}{1 - \kappa}$.

  We claim that if for some $i \neq 1$, there holds $b_{i} < \dfrac{-\kappa \varepsilon}{1 - \kappa}$, then $\mathcal{T}_{\Gamma_{\mathbf{b}}(m_{i})} \subseteq E$, where $E$ is the singular point set of $\Gamma(\mathbf{b})$.

  Indeed, if $x \in \mathcal{T}_{\Gamma_{\mathbf{b}}(m_{i})}$, then there exist $b>0$, such that $H(m,b)$ supports $\Gamma$ at $\rho(x_{0})x_{0}$. Then we have
  $$\rho(x) = \max_{1 \leq i \leq l}\frac{b_{i}}{1 - \kappa m_{i} \cdot x}, \quad \rho(x) \geq \frac{b}{1 - \kappa m_{i} \cdot x} \quad \text{and} \quad \rho(x_{0}) = \frac{b}{1 - \kappa m_{i} \cdot x_{0}},$$
  hence
  $$\frac{b_{i}}{1 - \kappa m_{i} \cdot x_{0}} \leq \rho(x_{0}) = \frac{b}{1 - \kappa m_{i} \cdot x_{0}},$$
  so we have $b_{i} \leq b$. If $b_{i} = b$, then $H(m_{i},b_{i})$ supports $\Gamma$, that is a contradiction, so $b_{i} < b$. Then for any $x \in \bar{\Omega}$, we have $\rho(x) \geq \dfrac{b}{1 - \kappa m_{i} \cdot x} > \dfrac{b_{i}}{1 - \kappa m_{i} \cdot x}$, so $\rho(x) = \max\limits_{j \neq 1} \dfrac{b_{j}}{1 - \kappa m_{j} \cdot x}$. Consequently, there exist $k \neq i$, such that
  $$\rho(x_{0}) = \frac{b_{k}}{1 - \kappa m_{k} \cdot x_{0}} \quad \text{and} \quad \rho(x) \geq \frac{b_{k}}{1 - \kappa m_{k} \cdot x}~\text{for all}~x\in \bar{\Omega}.$$
  So $x_{0} \in E$, then we have $\mathcal{T}_{\Gamma_{\mathbf{b}}(m_{i})} \subseteq E$.

  So we have
  $$G_{\Gamma(\mathbf{b})}(m_{i}) = \int_{\mathcal{T}_{\Gamma(\mathbf{b})}(m_{i})}f(x)t_{\Gamma(\mathbf{b})}(x)dx \leq  \int_{E}f(x)t_{\Gamma(\mathbf{b})}(x)dx < g_{i}.$$
  Take $\mathbf{b} = (1,b_{2}, \ldots, b_{l})$, such that $b_{i} < \dfrac{-\kappa \varepsilon}{1 - \kappa}$ for $2 \leq i \leq l$, then $\mathbf{b} \in W$, hence $W \neq \emptyset$.

  $(b)$ From Remark \ref{rem3.1}, we first claim that if $\mathbf{b} \in W$, then $g_{1} \leq G_{\Gamma(\mathbf{b})}(m_{1})$.

  Indeed, for we have
  \begin{align*}
     \sum_{i=1}^{l} G_{\Gamma(\mathbf{b})}(m_{i}) & =  \sum_{i=1}^{l} \int_{\mathcal{T}_{\Gamma(\mathbf{b})}(m_{i})}f(x)t_{\Gamma(\mathbf{b})}(x)dx \\
     & =  \int_{\bigcup\limits_{i=1}^{l}\mathcal{T}_{\Gamma(\mathbf{b})}(m_{i})}f(x)t_{\Gamma(\mathbf{b})}(x)dx  =  \int_{\bar{\Omega}}f(x)t_{\Gamma(\mathbf{b})}(x)dx\\
     & \geq (1 - C_{\varepsilon})\int_{\bar{\Omega}}f(x)dx \geq \mu(\bar{\Omega}^{\ast}) \\
     & = \sum_{i = 1}^{l} g_{i}.
  \end{align*}
  So we have
  $$g_{1} - G_{\Gamma(\mathbf{b})}(m_{1}) + \sum_{i=2}^{l}[g_{i} - G_{\Gamma(\mathbf{b})}(m_{i})] \leq 0. $$
  If $\mathbf{b} \in W$, then we have $g_{1} \leq G_{\Gamma(\mathbf{b})}(m_{1}).$

  Let $\Gamma(\mathbf{b}) = \{\rho(x)x;~x \in \bar{\Omega}\}$, we claim that there exists $\rho(x_{0})x_{0}$, such that $\rho(x_{0})x_{0} \in \Gamma(\mathbf{b}) \cap H(m_{1},1)$ and $\rho(x_{0})x_{0} \notin H(m_{i},b_{i})$ for all $i \geq 2$.

  Indeed, if not, we have $\mathcal{T}_{\Gamma(\mathbf{b})}(m_{1}) \subseteq E$, then
  $$G_{\Gamma(\mathbf{b})}(m_{1}) = \int_{\mathcal{T}_{\Gamma(\mathbf{b})}(m_{1})} f(x)t_{\Gamma(\mathbf{b})}(x)dx \leq \int_{E} f(x)t_{\Gamma(\mathbf{b})}(x)dx = 0$$
  for $\vert E \vert =0$. This is a contradiction with $g_{1} >0$, hence
  $$\rho(x_{0}) = \frac{1}{1 - \kappa m_{1} \cdot x_{0}} > \frac{b_{i}}{1 - \kappa m_{i} \cdot x_{0}},$$
  and thus we have
  $$b_{i} < \frac{1 - \kappa m_{i} \cdot x_{0}}{1 - \kappa m_{1} \cdot x_{0}} < \frac{1 - \kappa}{-\kappa \varepsilon}.$$
\end{proof}

\begin{lemma}
  Let $\mathbf{b}_{k} = (b_{1}^{k},\ldots,b_{l}^{k})$ and $\mathbf{b}_{0} = (b_{1}^{0},\ldots,b_{l}^{0})$ with $\mathbf{b}_{k}\rightarrow \mathbf{b}_{0}$ in $\mathbb{R}^{l}$. Suppose that $\Gamma_{k} = \Gamma(\mathbf{b}_{k}) = \{\rho_{k}(x)x;~x\in \bar{\Omega}\}$, $\Gamma_{0} = \Gamma(\mathbf{b}_{0}) = \{\rho(x)x;~x\in \bar{\Omega}\}$, then $\rho_{k} \rightarrow \rho$ uniformly on $\bar{\Omega}$.
  \label{lem3.9}
\end{lemma}

\begin{proof}
  For $x_{0} \in \bar{\Omega}$, we have
  \begin{align*}
    \vert \rho(x_{0}) - \rho_{k}(x_{0}) \vert & = \vert \frac{b_{i}}{1 - \kappa m_{i} \cdot x_{0}} - \rho_{k}(x) \vert \quad \text{for some}~i  \\
     & \leq  \vert \frac{b_{i}}{1 - \kappa m_{i} \cdot x_{0}} - \frac{b_{i}^{k}}{1 - \kappa m_{i} \cdot x_{0}} \vert \\
     & \leq \frac{\Vert \mathbf{b} - \mathbf{b}^{k} \Vert}{-\kappa \varepsilon},
  \end{align*}
  hence $\rho_{k}\rightarrow \rho$ uniformly on $\bar{\Omega}$.
\end{proof}

\begin{lemma}
  Let $\tau >0$, then $G_{\Gamma(\mathbf{b})}(m_{i}) = \displaystyle\int_{\mathcal{T}_{\Gamma(\mathbf{b})}(m_{i})}f(x)t_{\Gamma(\mathbf{b})}(x)dx$ is continuous on the region $R_{\tau} = \{(1,b_{2},\ldots, b_{l});~0<b_{i}\leq \tau,~i=2,\ldots,l\}$, for any $1 \leq i \leq l$.
  \label{lem3.10}
\end{lemma}

\begin{proof}
  Suppose that $\mathbf{b}_{k}$ is a sequence converges to $\mathbf{b}_{0}$ in $R_{\tau}$, and let $\Gamma(\mathbf{b}_{k}) = \{\rho_{k}(x)x;~x\in \bar{\Omega}\}$, $\Gamma(\mathbf{b}_{0}) = \{\rho_{0}(x)x;~x\in \bar{\Omega}\}$. Then from Lemma \ref{lem3.9}, $\rho_{k} \rightarrow \rho$ uniformly on $\bar{\Omega}$. Besides, for any $x\in \bar{\Omega}$ and $k \geq 1$, we have
  $$\rho_{k}(x) = \frac{b_{i}^{k}}{1 - \kappa m_{i} \cdot x} \leq \max\{\dfrac{\tau}{-\kappa \varepsilon}, \dfrac{1}{-\kappa \varepsilon}\}$$
  and
  $$\rho_{k}(x) = \max_{1 \leq i \leq l}\frac{b_{i}^{k}}{1 - \kappa m_{i} \cdot x} \geq \frac{1}{1 - \kappa m_{1} \cdot x} \geq\frac{1}{1 - \kappa}.$$
  for some $i \in \{1,2,\ldots,l\}$. Hence there exist $0 < a_{1} \leq a_{2}$, such that $ a_{1} \leq \rho_{k}(x) \leq  a_{2}$.

  Suppose that $G \subseteq \bar{\Omega}^{\ast}$ is a neighborhood of $m_{i}$, such that $m_{j} \notin G$ for all $j \neq i$. If $x_{0} \in \mathcal{T}_{\Gamma({\mathbf{b}_{k}})}(G)$ and $x_{0} \notin E$, then there exists a unique $m \in G$ and $b >0$, such that
  $$\rho_{k}(x_{0}) = \frac{b}{1 - \kappa m \cdot x_{0}} \quad \text{and} \quad \rho_{k}(x) \geq \frac{b}{1 - \kappa m \cdot x}~\text{for all}~x \in \bar{\Omega}.$$
  From the definition of $\Gamma(\mathbf{b}_{k})$ in Lemma \ref{lem3.9}, we have $m = m_{j}$ for some $j = 1,2,\ldots,l$, hence we have $m = m_{j}$, then $\mathcal{T}_{\Gamma({\mathbf{b}_{k}})}(G) \subseteq \mathcal{T}_{\Gamma({\mathbf{b}_{k}})}(m_{i}) \cup E$. For $\vert E \vert = 0$, from Lemma \ref{lem3.6}, we have
  \begin{equation}\label{3.11}
    \begin{aligned}
    \int_{\mathcal{T}_{\Gamma({\mathbf{b}_{0}})}(G)}f(x)t_{\Gamma({\mathbf{b}_{0}})}(x)dx & \leq \int_{\varliminf\limits_{k \rightarrow \infty} \mathcal{T}_{\Gamma({\mathbf{b}_{k}})}(m_{i}) \cup E}f(x)t_{\Gamma({\mathbf{b}_{0}})}(x)dx \\
     & \leq \int_{\varliminf\limits_{k \rightarrow \infty} \mathcal{T}_{\Gamma({\mathbf{b}_{k}})}(m_{i})}f(x)t_{\Gamma({\mathbf{b}_{0}})}(x)dx + \int_{E}f(x)t_{\Gamma({\mathbf{b}_{0}})}(x)dx \\
     & = \int_{\bar{\Omega}}\chi_{\varliminf\limits_{k \rightarrow \infty} \mathcal{T}_{\Gamma({\mathbf{b}_{k}})}(m_{i})}f(x)t_{\Gamma({\mathbf{b}_{0}})}(x)dx.
  \end{aligned}
  \end{equation}
  Obviously, we have
  \begin{equation}\label{3.12}
   \chi_{\varliminf\limits_{k \rightarrow \infty} \mathcal{T}_{\Gamma({\mathbf{b}_{k}})}(m_{i})}(x) = \varliminf\limits_{k \rightarrow \infty} \chi_{\mathcal{T}_{\Gamma({\mathbf{b}_{k}})}(m_{i})}(x).
  \end{equation}
  Applying Theorem \ref{thm3.1}, (\ref{3.12}) and Fatou lemma to (\ref{3.11}), we have
  \begin{equation}\label{3.13}
    \begin{aligned}
    \int_{\mathcal{T}_{\Gamma({\mathbf{b}_{0}})}(G)}f(x)t_{\Gamma({\mathbf{b}_{0}})}(x)dx & \leq \int_{\bar{\Omega}}\varliminf\limits_{k \rightarrow \infty} \chi_{\mathcal{T}_{\Gamma({\mathbf{b}_{k}})}(m_{i})}(x)t_{\Gamma({\mathbf{b}_{k}})}f(x)dx \\
    & \leq \varliminf\limits_{k \rightarrow \infty} \int_{\bar{\Omega}}\chi_{\mathcal{T}_{\Gamma({\mathbf{b}_{k}})}(m_{i})}(x)t_{\Gamma({\mathbf{b}_{k}})}f(x)dx \\
    & = \varliminf\limits_{k \rightarrow \infty} \int_{\mathcal{T}_{\Gamma({\mathbf{b}_{k}})}(m_{i})}t_{\Gamma({\mathbf{b}_{k}})}f(x)dx.
    \end{aligned}
  \end{equation}
  Besides, we also have
  \begin{equation}\label{3.14}
    \chi_{\varlimsup\limits_{k \rightarrow \infty} \mathcal{T}_{\Gamma({\mathbf{b}_{k}})}(m_{i})}(x) = \varlimsup\limits_{k \rightarrow \infty} \chi_{\mathcal{T}_{\Gamma({\mathbf{b}_{k}})}(m_{i})}(x).
  \end{equation}
  From inverse Fatou lemma, Lemma \ref{lem3.6}, Theorem \ref{thm3.1} and (\ref{3.14}), we have
  \begin{equation}\label{3.15}
    \begin{aligned}
     \varlimsup\limits_{k \rightarrow \infty} \int_{\mathcal{T}_{\Gamma({\mathbf{b}_{k}})}(m_{i})}t_{\Gamma({\mathbf{b}_{k}})}f(x)dx  & \leq \int_{\bar{\Omega}}\varlimsup\limits_{k \rightarrow \infty} \chi_{\mathcal{T}_{\Gamma({\mathbf{b}_{k}})}(m_{i})}(x)t_{\Gamma({\mathbf{b}_{k}})}f(x)dx  \\
     & = \int_{\bar{\Omega}}\chi_{\varlimsup\limits_{k \rightarrow \infty} \mathcal{T}_{\Gamma({\mathbf{b}_{k}})}(m_{i})}(x)f(x)t_{\Gamma(\mathbf{b}_{0})}(x)dx \\
     & = \int_{\varlimsup\limits_{k \rightarrow \infty} \mathcal{T}_{\Gamma({\mathbf{b}_{k}})}(m_{i})}f(x)t_{\Gamma(\mathbf{b}_{0})}(x)dx \\
     & \leq \int_{\mathcal{T}_{\Gamma({\mathbf{b}_{0}})}(G)}f(x)t_{\Gamma(\mathbf{b}_{0})}(x)dx.
    \end{aligned}
  \end{equation}
  Combining (\ref{3.13}) with (\ref{3.15}), we get $G_{\Gamma(\mathbf{b})}(m_{i})$ is continuous on the region $R_{\tau}$.
\end{proof}

Based on the above lemmas, now we prove the existence of the weak solution.

\begin{proof}[Proof of Theorem \ref{thm3.3}]
  Fixed $\overline{\mathbf{b}} = (1,\overline{b_{2}},\ldots,\overline{b_{l}})$, consider the set $\overline{W} = \{\mathbf{b}_{i} = (1,b_{2},\ldots,b_{l});
  \newline b_{i} \leq \overline{b_{i}}, i=2,\ldots,l\}$, then from Lemma \ref{lem3.8} and Lemma \ref{lem3.10}, $W$ is a compact set. Define a mapping
  $$d:\overline{W}\rightarrow \mathbb{R};~\mathbf{b} \mapsto \sum_{i=1}^{l}b_{i}.$$
  Let $\mathbf{b}^{\ast} = \arg\max\limits_{\mathbf{b}\in \overline{W}}d(\mathbf{b})$, for the compactness of $\overline{W}$, then we know $d$ is a continuous mapping hence $\mathbf{b}^{\ast}$ exists.

  Taking $\mathbf{b}_{0} = \mathbf{b}^{\ast}$, we first prove that $\displaystyle\int_{\mathcal{T}_{\Gamma(\mathbf{b}_{0})}(m_{i})}f(x)t_{\Gamma(\mathbf{b}_{0})}(x)dx = g_{i}$ for $i = 2,\ldots,l$.

  Indeed, if not, we may assume that $\displaystyle\int_{\mathcal{T}_{\Gamma(\mathbf{b}_{0})}(m_{2})}f(x)t_{\Gamma(\mathbf{b}_{0})}(x)dx < g_{2}.$ Taking $\xi >1$ and let $\mathbf{b}_{\xi} = (1,\xi b_{2}^{\ast},\ldots, b_{l}^{\ast})$. If $x_{0} \in \mathcal{T}_{\Gamma(\mathbf{b}_{\xi})}(m_{i}) \setminus E_{\xi}^{\ast}$, where $E_{\xi}^{\ast}$ is the singular point set of $\mathbf{b}_{\xi}$, then we have
  $$\rho(x_{0}) = \frac{b_{i}^{\ast}}{1 - \kappa m_{i} \cdot x_{0}} \quad \text{and} \quad \rho(x) \geq \frac{b_{i}^{\ast}}{1 - \kappa m_{i} \cdot x}~\text{for all}~x\in \bar{\Omega},$$
  hence $x_{0} \in \mathcal{T}_{\Gamma(\mathbf{b}^{\ast})}(m_{i})$, then $\mathcal{T}_{\Gamma(\mathbf{b}_{\xi}^{\ast})}(m_{i}) \setminus E_{\xi}^{\ast} \subseteq \mathcal{T}_{\Gamma(\mathbf{b}^{\ast})}(m_{i})$. So we have
  $$\int_{\mathcal{T}_{\Gamma(\mathbf{b}_{\xi}^{\ast})}(m_{i})}f(x)t_{\Gamma(\mathbf{b}_{\xi}^{\ast})}dx = \int_{\mathcal{T}_{\Gamma(\mathbf{b}_{\xi}^{\ast})}(m_{i})}f(x)t_{\Gamma(\mathbf{b}^{\ast})}dx \leq \int_{\mathcal{T}_{\Gamma(\mathbf{b}^{\ast})}(m_{i})}f(x)t_{\Gamma(\mathbf{b}^{\ast})}dx.$$
  Let $\xi \rightarrow 1$, then from Lemma \ref{lem3.10}, we have $G_{\Gamma_{\mathbf{b}_{\xi}^{\ast}}} < g_{2}$, hence $b_{\xi}^{\ast} \in W$, this is a contradiction with $d(\mathbf{b}_{\xi}^{\ast}) \leq d(\mathbf{b}_{\xi})$.
  
  Now we prove that $\displaystyle\int_{\mathcal{T}_{\Gamma(\mathbf{b}_{0})}(m_{1})}f(x)t_{\Gamma(\mathbf{b}_{0})}(x)dx > g_{1}$.
  
  Indeed, from Lemma \ref{lem3.8}, we have $\displaystyle\int_{\mathcal{T}_{\Gamma(\mathbf{b}_{0})}(m_{1})}f(x)t_{\Gamma(\mathbf{b}_{0})}(x)dx \geq g_{1}$. If the equality holds, then we have
  $$\int_{\bar{\Omega}}f(x)t_{\Gamma(\mathbf{b}_{0})}(x)dx = \sum_{i = 1}^{l}g_{i} \leq (1 - C_{\varepsilon})\int_{\bar{\Omega}}f(x)dx,$$
  hence 
  $$\int_{\bar{\Omega}}f(x)[1 - C_{\varepsilon} - t_{\Gamma(\mathbf{b}_{0})}(x)]dx \geq 0.$$
  From (\ref{2.17}), we have $t_{\Gamma(\mathbf{b}_{0})}(x) \geq 1 - C_{\varepsilon}$. But for $\inf\limits_{x \in \bar{\Omega}}f(x)>0$, then we must have $t_{\Gamma(\mathbf{b}_{0})}(x) = 1 - C_{\varepsilon}$ for a.e $x \in \bar{\Omega}$. From (\ref{3.5}), for $x \in \mathcal{T}_{\Gamma(\mathbf{b}_{0})}(m_{1}) \backslash E$, we have $\psi(x \cdot m_{1}) = C_{\varepsilon}$, then $\vert \mathcal{T}_{\Gamma(\mathbf{b}_{0})}(m_{1}) \backslash E \vert > 0$.
  
  We claim that the set $D = \{x \cdot m_{1};~x \in \mathcal{T}_{\Gamma(\mathbf{b}_{0})}(m_{1})\}$ is infinite.
  
  Indeed, if not, then there exist $c_{1}, \ldots, c_{n}$, such that $D = {c_{1}, \ldots, c_{n}}$. Let $D_{j} = \{x \in \mathcal{T}_{\Gamma(\mathbf{b}_{0})};~x\cdot m_{1} = c_{j}\}$, then $D = \bigcup\limits_{j=1}^{n} D_{j}$. But $D_{j}$ contains in $S^{n-1}$ intersected with the plane $\{x;~x\cdot m_{1} = c_{j}\}$, hence its spherical measure is 0, then $\vert \mathcal{T}_{\Gamma(\mathbf{b}_{0})}(m_{1}) \backslash E \vert = 0$. This is a contradiction, hence $D = \{x \cdot m_{1};~x \in \mathcal{T}_{\Gamma(\mathbf{b}_{0})}(m_{1})\}$ is infinite. Besides, from Proposition \ref{prop3.1}, we know that the set $\{t;~\psi(t) = c\}$ is a finite set for any constant $c$, then we cannot have $\psi = C_{\varepsilon}$ on $D$. So we must have $\displaystyle\int_{\mathcal{T}_{\Gamma(\mathbf{b}_{0})}(m_{1})}f(x)t_{\Gamma(\mathbf{b}_{0})}(x)dx > g_{1}$.
\end{proof}

\subsection{Existence of the weak solution when $\mu$ is a finite Radon measure}

In this subsection, we assume that $\mu$ is a finite Radon measure, and the existence of the weak solution of the far field refraction problem for the case $\kappa <-1$ with loss of energy in this situation is established by using discrete measures to approximate.

\begin{theorem}
  Suppose that $f$ is integrable on $\bar{\Omega}$ and $\inf\limits_{x\in \bar{\Omega}}f(x)>0$. Let $\mu$ be a Radon measure on $\bar{\Omega}^{\ast}$ and
  \begin{equation}\label{3.16}
    \int_{\bar{\Omega}}f(x)dx \geq \frac{1}{1 - C_{\varepsilon}}\mu(\bar{\Omega}^{\ast}),
  \end{equation}
  where $C_{\varepsilon}$ is defined in Proposition \ref{prop3.1}. Then there exists a refractor $\Gamma$, such that for any Borel subset $\omega \subseteq \bar{\Omega}^{\ast}$, we have
  $$\mu(\omega) \leq \int_{\mathcal{T}_{\Gamma}(\omega)}f(x)t_{\Gamma}(x)dx,$$
  that is, there exists a weak solution of the refraction problem for the case $\kappa < -1$ with emitting illumination intensity $f$ and prescribed refracted intensity $\mu$.
  \label{thm3.4}
\end{theorem}

\begin{proof}
  Let $\iota$ be an integer, $\iota \geq 2$. Segmenting $\bar{\Omega}^{\ast}$ into finite disjoint subsets $\omega_{1}^{\iota}, \omega_{2}^{\iota}, \ldots,\omega_{l_{\iota}}^{\iota}$, such that $diam(\omega_{i}^{\iota}) \leq \dfrac{1}{\iota}$ for $i = 1,2,\ldots,l_{\iota}$. Take $m_{i}^{\iota} \in \omega_{i}^{\iota}$ and consider the measure $\mu_{\iota} := \sum\limits_{i = 1}^{l _{\iota}}\mu(\omega_{i}^{\iota}) \delta_{m_{i}^{\iota}}$ defined on $\bar{\Omega}^{\ast}$.

  We claim that $\mu_{\iota} \rightarrow \mu$ weakly as $\iota \rightarrow \infty$.

  Indeed, take $h \in C(\bar{\Omega}^{\ast})$, then we have
  \begin{align*}
    \int_{\bar{\Omega}^{\ast}}h d\mu_{\iota} - \int_{\bar{\Omega}^{\ast}}h d\mu & = \sum_{i = 1}^{l_{\iota}}\int_{\bar{\Omega}^{\ast}}h\mu(\omega_{i}^{\iota})d\delta_{m_{i}^{\iota}} - \int_{\bar{\Omega}^{\ast}}h d\mu \\
      & = \sum_{i = 1}^{l_{\iota}} \int_{\omega_{i}^{\iota}} h(m_{i}^{\iota})d\mu - \sum_{i = 1}^{l_{\iota}}\int_{\omega_{i}^{\iota}}h(x)d\mu 
        = \sum_{i = 1}^{l_{\iota}} \int_{\omega_{i}^{\iota}}(h(m_{i}^{\iota}) - h(x))d\mu.
  \end{align*}
  For $h \in C(\bar{\Omega}^{\ast})$ and $diam(\omega_{i}^{\iota}) \leq \dfrac{1}{\iota}$, hence $\displaystyle\int_{\bar{\Omega}^{\ast}}hd\mu_{\iota} \rightarrow \displaystyle\int_{\bar{\Omega}^{\ast}}hd\mu$ as $\iota \rightarrow \infty$. Consequently, $\mu_{\iota} \rightarrow \mu$ weakly as $\iota \rightarrow \infty$.

  From (\ref{3.16}), we have $\mu_{\iota}(\bar{\Omega}^{\ast}) = \mu(\bar{\Omega}^{\ast}) \leq (1 - C_{\varepsilon})\displaystyle\int_{\bar{\Omega}}f(x)dx$, then from Theorem \ref{thm3.3}, there exists a refractor $\Gamma_{\iota} = \{\rho_{\iota}(x)x;~\rho_{\iota}(x) = \max\limits_{1 \leq i \leq l_{\iota}}\dfrac{b_{i}}{1 - \kappa m_{i}^{\iota} \cdot x}\}$, such that $\mu_{\iota}(\omega) = \displaystyle\int_{\mathcal{T}_{\Gamma_{\iota}}(\omega)}f(x)t_{\Gamma_{\iota}}(x)dx$. Normalized $\Gamma_{\iota}$, such that $\inf\limits_{x \in \bar{\Omega}}\rho_{\iota}(x) = 1$, then from Lemma \ref{lem3.2}, there exists a constant $C>0$, such that $\sup\limits_{x \in \bar{\Omega}}\rho_{\iota}(x) \leq C$ for all $\iota \geq 1$.

  Besides, if $x_{0}, x_{1} \in \bar{\Omega}$ and $H(m_{0},b_{0})$ supports $\Gamma_{\iota}$ at $\rho_{\iota}(x_{0})x_{0}$, then for $x_{1} \in \bar{\Omega}$, we have
  \begin{align*}
    \vert \rho_{\iota}(x_{0}) - \rho_{\iota}(x_{1}) \vert & \leq \vert \frac{b_{0}}{1 - \kappa m_{0} \cdot x_{0}} - \frac{b_{0}}{1 - \kappa m_{0} \cdot x_{1}} \vert \\
      & \leq \frac{\kappa b_{0}}{(1 - \kappa m_{0} \cdot x_{0})(1 - \kappa m_{0} \cdot x_{1})} \Vert x_{0} - x_{1} \Vert \\
      & \leq \frac{b_{0}}{-\kappa \varepsilon^{2}} \Vert x_{0} - x_{1} \Vert  
        \leq \frac{C}{-\kappa \varepsilon^{2}} \Vert x_{0} - x_{1} \Vert.
  \end{align*}
  Exchanging the roles of $x_{0}$ and $x_{1}$, we have
  $$\vert \rho_{\iota}(x_{1}) - \rho_{\iota}(x_{0}) \vert \leq \frac{C}{-\kappa \varepsilon^{2}} \Vert x_{1} - x_{0} \Vert,$$
  hence $\{\rho_{\iota}(x);~\iota \geq 1\}$ is a family of bounded uniformly and equicontinuous functions. Then from Arezlà-Ascoli Theorem, $\rho_{\iota}(x) \rightarrow \rho(x)$ uniformly as $\iota \rightarrow \infty$ for all $x \in \bar{\Omega}$. Then from Lemma \ref{lem3.6} $(a)$, $\Gamma = \{\rho(x)x;~x \in \bar{\Omega}\}$ is a refractor. 
  
  Let $G_{\Gamma{\iota}}(\omega) := \displaystyle \int_{\mathcal{T}_{\Gamma_{\iota}}(\omega)}f(x)t_{\Gamma_{\iota}}(x)dx$ and $G_{\Gamma}(\omega) := \displaystyle \int_{\mathcal{T}_{\Gamma}(\omega)}f(x)t_{\Gamma}(x)dx$. In order to prove the existence of the weak solution, we still need to prove that $G_{\Gamma{\iota}} \rightarrow G_{\Gamma}$ weakly as $\iota \rightarrow \infty$.
  
  Indeed, on the one hand, for any compact set $K \subseteq \bar{\Omega}^{\ast}$, from inverse Fatou lemma, we have
  \begin{align*}
    \varlimsup_{\iota \rightarrow \infty} G_{\Gamma{\iota}}(K) & = \varlimsup_{\iota \rightarrow \infty} \int_{\mathcal{T}_{\Gamma_{\iota}}(K)}f(x)t_{\Gamma_{\iota}}(x)dx \\
      & \leq \int_{\bar{\Omega}} \varlimsup_{\iota \rightarrow \infty} \chi_{\mathcal{T}_{\Gamma_{\iota}}(K)} f(x)t_{\Gamma_{\iota}}(x)dx\\
      & \leq \int_{\mathcal{T}_{\Gamma}(K)}f(x)t_{\Gamma}(x)dx = G_{\Gamma}(K).
  \end{align*}
  On the other hand, for any open set $F \subseteq \bar{\Omega}^{\ast}$, from Fatou lemma, we have
  \begin{align*}
    G_{\Gamma}(F) & = \int_{\mathcal{T}_{\Gamma}(F)}f(x)t_{\Gamma}(x)dx \\
     & \leq \int_{\bar{\Omega}}\varliminf_{\iota \rightarrow \infty}\chi_{\mathcal{T}_{\Gamma_{\iota}}(F)}f(x)t_{\Gamma}(x)dx \\
     & \leq \varliminf_{\iota \rightarrow \infty} \int_{\bar{\Omega}}\chi_{\mathcal{T}_{\Gamma_{\iota}}(F)}f(x)t_{\Gamma}(x)dx = \varliminf_{\iota \rightarrow \infty} G_{\Gamma{\iota}}(F).
  \end{align*}
  Consequently, we have $G_{\Gamma{\iota}} \rightarrow G_{\Gamma}$ weakly as $\iota \rightarrow \infty$, hence we have proved the existence of the weak solution.
\end{proof}

\section{Far field refraction problem for the case $-1 <\kappa <0$ with loss of energy}\label{Section 4}

\sloppy{}

In this section, we use the similar way as Section \ref{Section 3} to study the far field refraction problem for the case $-1 <\kappa <0$ with loss of energy. Recall from (\ref{2.7}) in Remark \ref{rem2.1}, we mast have
\begin{equation*}
  \inf_{x \in \bar{\Omega}, m \in \bar{\Omega}^{\ast}} x \cdot m \geq \kappa + \varepsilon
\end{equation*}
in this case.

\subsection{Refractor and its properties}

The definition of refractor in the case $-1 < \kappa<0$ also stems from \cite{St17}.

\begin{definition}
  A parameterized surface $\Gamma$ in $\mathbb{R}^{n}$ given by $\Gamma = \{\rho(x)x; ~\rho \in C(\bar{\Omega})\}$ is a refractor from $\bar{\Omega}$ to $\bar{\Omega}^{\ast}$ in the case $-1 < \kappa<0$, if for any $x_{0} \in \bar{\Omega}$, there exists a semi-ellipsoid defined as $E(m,b) = \{\rho(x)x; ~\rho(x) = \dfrac{b}{1 - \kappa m \cdot x}, ~x\in S^{n-1}, ~x \cdot m \geq \kappa\}$, such that $\rho(x_{0}) = \dfrac{b}{1 - \kappa m\cdot x_{0}}$ and $\rho(x) \leq \dfrac{b}{1 - \kappa m\cdot x}$ for all $x\in \bar{\Omega}$. Such $E(m,b)$ is called a supporting ellipsoid to $\Gamma$ at the point $\rho(x_{0})x_{0}$.
  \label{def4.1}
\end{definition}

The following Figure \ref{fig4} shows the semi-ellipsoid which refracts all ray emitted from the source $O$ to a specific direction for the case $-1<\kappa<0$.

\begin{figure}[h]
  \includegraphics[width=8.3cm]{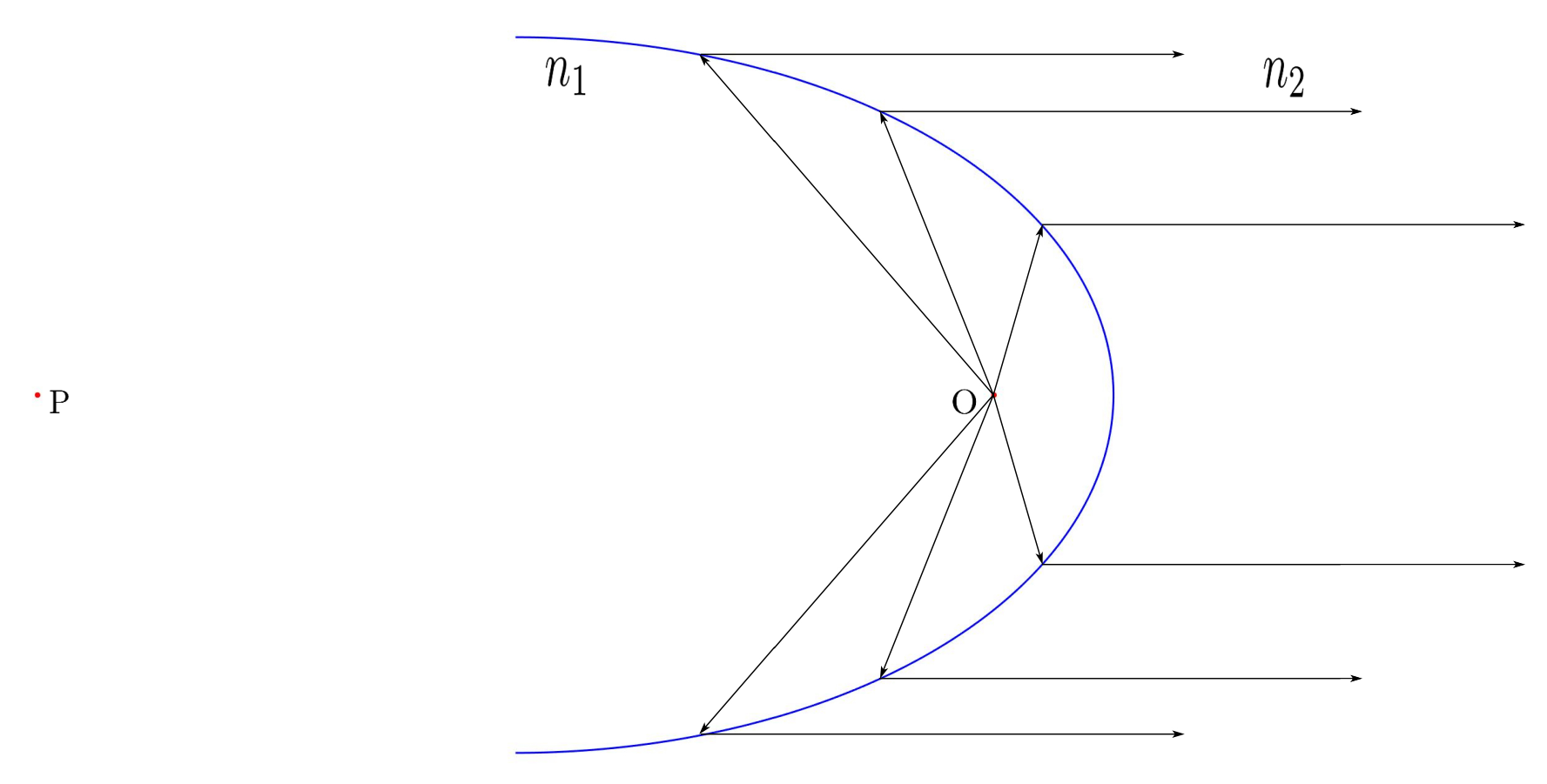}
\caption{Ellipsoid refracting when $-1<\kappa<0$, where $O$ and $P$ are focus of ellipsoid.}\label{fig4}
\end{figure}

The following lemmas are similar as Lemmas \ref{lem3.1} and   \ref{lem3.2}.

\begin{lemma}
  Any refractor is globally Lipschitz continuous on $\bar{\Omega}$, hence the set of singular points is a null set.
  \label{lem4.1}
\end{lemma}

\begin{lemma}
  Suppose $\Gamma = \{\rho(x)x ; ~\rho \in C(\bar{\Omega})\}$ is a refractor from $\bar{\Omega}$ to $\bar{\Omega}^{\ast}$, such that $\inf\limits_{x\in \bar{\Omega}}\rho(x) = 1$, then there exists a constant $C>0$ depending on $\kappa$, such that $\sup\limits_{x\in \bar{\Omega}}\rho(x) \leq C$.
  \label{lem4.2}
\end{lemma}

\begin{proof}
    Suppose that there exists $x_{0} \in \bar{\Omega}$, such that $\rho(x_{0}) = \inf\limits_{x\in \bar{\Omega}}\rho(x)$, and let $E(m_{0},b_{0})$ be the supporting ellipsoid to $\Gamma$ at $\rho(x_{0})x_{0}$. Then we have
    $1 = \rho(x_{0}) = \dfrac{b_{0}}{1 - \kappa m_{0} \cdot x_{0}},$ and $\rho(x) \leq \dfrac{b_{0}}{1 - \kappa m_{0} \cdot x}.$
    Hence we have $b_{0} = 1- \kappa m_{0} \cdot x_{0} \leq 1 - \kappa$.
    Consequently,
    $$\rho(x) \leq \frac{b_{0}}{1 - \kappa m_{0} \cdot x} \leq \frac{1 - \kappa}{1 - \kappa^{2}} \leq \frac{1}{1 + \kappa}.$$
    Then we get $\sup\limits_{x\in \bar{\Omega}}\rho(x) \leq C$.
\end{proof}

\begin{remark}
  Compared with Lemma \ref{lem3.2}, in this case, the constant $C$ only depends on $\kappa$.
  \label{rem4.1}
\end{remark}

\begin{remark}
  If a refractor $\Gamma$ parameterized by $\rho$ has two distinct supporting semi-ellipsoid at $\rho(x)x$, then $\rho(x)x$ is a singular point of $\Gamma$.
  \label{rem4.2}
\end{remark}

We can also define refractor mapping and trace mapping for the case $-1 < \kappa <0$ and discuss some properties of them.

\begin{definition}
  Suppose that the refractor $\Gamma = \{\rho(x)x;~x \in \bar{\Omega}\}$ is given, the refractor mapping of $\Gamma$ is a multi-value map defined by
  \begin{equation}\label{4.1}
    \mathcal{N}_{\Gamma}(x_{0}) = \{m \in \bar{\Omega}^{\ast};~E(m,b) ~supports~\Gamma ~at ~\rho(x_{0})x_{0} ~ for ~ some ~ b>0 \}.
  \end{equation}
  Given $m \in \bar{\Omega}^{\ast}$, the trace mapping of $\Gamma$ is defined by
  \begin{equation}\label{4.2}
    \mathcal{T}_{\Gamma}(m_{0}) = \{x \in \bar{\Omega};~m_{0}\in \mathcal{N}_{\Gamma}(x_{0})\}.
  \end{equation}
  \label{def4.2}
\end{definition}

The proofs of the following properties are analogous to those in Section  \ref{sec3.1}.

\begin{lemma}
  If $m \in \bar{\Omega}^{\ast}$, then $\mathcal{T}_{\Gamma}(m)$ is a closed set in $\bar{\Omega}$.
  \label{lem4.3}
\end{lemma}

\begin{lemma}
  For any $F \in \bar{\Omega}^{\ast}$, we have
  
   (a) $[\mathcal{T}_{\Gamma}(F)]^{c} \subseteq \mathcal{T}_{\Gamma}(F^{c})$;
   
   (b) The set $\mathcal{M} = \{F \subseteq \bar{\Omega}^{\ast};~\mathcal{T}_{\Gamma}(F)~is~Lebesgue~measurable\}$ is a $\sigma$-algebra containing all Borel sets in $\bar{\Omega}^{\ast}$.
  \label{lem4.4}
\end{lemma}

\begin{lemma}
  Suppose that $E(m_{k},b_{k})$ is a sequence of semi-ellipsoid, and $m_{k}\rightarrow m_{0}$, $b_{k}\rightarrow b_{0}$ as $k\rightarrow \infty$. Let $z_{k} \in E(m_{k},b_{k})$ with $z_{k} \rightarrow z_{0}$ as $k\rightarrow \infty$. Then $z_{0} \in E(m_{0},b_{0})$, and the normal $\nu_{k}(z_{k})$ to the semi-ellipsoid $E(m_{k},b_{k})$ at $z_{k}$ satisfies $\nu_{k}(z_{k}) \rightarrow \nu(z_{0})$ the normal to the semi-ellipsoid $E(m_{0},b_{0})$ at the point $z_{0}$.
  \label{lem4.5}
\end{lemma}

\begin{lemma}
  Assume that $\Gamma_{k} = \{\rho_{k}(x)x;~x\in \bar{\Omega}\}$, $k \geq 1$ is a sequence of refractors from $\bar{\Omega}$ to $\bar{\Omega}^{\ast}$. Suppose that $0<a_{1}\leq \rho_{k} \leq a_{2}$ and $\rho_{k} \rightarrow \rho$ uniformly on $\bar{\Omega}$. Then we have
  
   (a) $\Gamma := \{\rho(x)x;~x\in \bar{\Omega}\}$ is a refractor from $\bar{\Omega}$ to $\bar{\Omega}^{\ast}$;
  
   (b) For any compact set $K\subseteq \bar{\Omega}^{\ast}$,
  $$\varlimsup_{k\rightarrow \infty} \mathcal{T}_{\Gamma_{k}}(K) \subseteq \mathcal{T}_{\Gamma}(K);$$
  
   (c) For any open set $G \subseteq \bar{\Omega}^{\ast}$,
  $$\mathcal{T}_{\Gamma}(G) \subseteq \varliminf_{k\rightarrow \infty}\mathcal{T}_{\Gamma_{k}}(G) \cup E,$$
  where $E$ is the singular set of $\Gamma$.
  \label{lem4.6}
\end{lemma}

\subsection{Properties of Fresnel coefficients}

Recall again the Fresnel coefficients in (\ref{2.16}) and (\ref{2.17}),
\begin{equation*}
  \begin{cases}
    r_{\Gamma}(x) \!\!\!\!&= \left[\dfrac{z_{2} + \kappa z_{1} - (z_{1} + \kappa z_{2}) x \cdot m}{z_{2} - \kappa z_{1} + (z_{1} - \kappa z_{2}) x \cdot m} \right]^{2} \dfrac{A_{\parallel}^{2}}{A_{\parallel}^{2} + A_{\bot}^{2}} \\
    \!\!\!\!&\quad+ \left[ \dfrac{z_{1} + \kappa z_{2} - (z_{2} + \kappa z_{1}) x \cdot m}{z_{1} - \kappa z_{2} + (z_{2} - \kappa z_{1}) x \cdot m}\right]^{2} \dfrac{A_{\bot}^{2}}{A_{\parallel}^{2} + A_{\bot}^{2}},  \\
    t_{\Gamma}(x) \!\!\!\!&= 1 - r_{\Gamma}(x).
  \end{cases}
\end{equation*}
We can also discuss the boundedness of $r_{\Gamma}(x)$ and $t_{\Gamma}(x)$.

\begin{proposition}
  Suppose that $\bar{\Omega}$ and $\bar{\Omega}^{\ast}$ satisfy (\ref{2.7}), $\Gamma$ is a refractor from $\bar{\Omega}$ to $\bar{\Omega}^{\ast}$, then there exists a constant $C_{\varepsilon}$ associated with $\varepsilon$, such that $r_{\Gamma}(x) \leq C_{\varepsilon}$ and $C_{\varepsilon} < t_{\Gamma}(x) <1$.
  \label{prop4.1}
\end{proposition}

\begin{proof}
  Similar to the discussion in section 3.2, we introduce a function
  \begin{equation}\label{4.3}
    \psi(t) := \left[\frac{\sigma + \kappa -(1 + \kappa \sigma)t}{\sigma - \kappa + (1 - \kappa \sigma)t}\right]^{2} \alpha + \left[\frac{1 + \kappa \sigma - (\sigma + \kappa)t}{1 - \kappa \sigma + (\sigma - \kappa)t}\right]^{2} \beta,
  \end{equation}
  where $\sigma = \displaystyle \frac{z_{2}}{z_{1}} = \displaystyle \sqrt{\frac{\mu_{2}\epsilon_{1}}{\mu_{1}\epsilon_{2}}}$, $\alpha = \dfrac{A_{\parallel}^{2}}{A_{\parallel}^{2} + A_{\bot}^{2}}$ and $\beta = \dfrac{A_{\bot}^{2}}{A_{\parallel}^{2} + A_{\bot}^{2}}$. Then $r_{\Gamma}(x) = \psi(x \cdot m)$, $t_{\Gamma}(x) = 1- \psi(x \cdot m)$. From (\ref{2.7}), we know $t \in [\kappa + \varepsilon, 1]$. We denote
  $$p(t) = \frac{\sigma + \kappa -(1 + \kappa \sigma)t}{\sigma - \kappa + (1 - \kappa \sigma)t} \quad \text{and} \quad q(t) = \frac{1 + \kappa \sigma - (\sigma + \kappa)t}{1 - \kappa \sigma + (\sigma - \kappa)t}.$$
  
  For $p(t)$, we have $p'(t) = \displaystyle \frac{2 \sigma (\kappa^{2} - 1)}{[\sigma - \kappa + (1 - \kappa \sigma)t]^{2}}$.  For $-1< \kappa <0$, then $\kappa^{2} - 1 <0$, so $p(t)$ decreases on $[\kappa + \varepsilon,1]$. Hence 
  $$p^{2}(t)_{\max} = \max\{p^{2}(\kappa + \varepsilon), p^{2}(1)\}.$$
  We have 
  $$p^{2}(1) = \left[\dfrac{\sigma - 1}{\sigma + 1}\right]^{2},~~p^{2}(\kappa + \varepsilon) = \left[\dfrac{\sigma(1 - \kappa^{2}) - \varepsilon (1 + \kappa \sigma)}{\sigma(1 - \kappa^{2}) + \varepsilon (1 - \kappa \sigma)}\right]^{2}.$$
   For $\varepsilon$ is small enough, then $p^{2}(t)_{\max} = p^{2}(\kappa + \varepsilon)$.
  
  For $q(t)$, we have $q'(t) = \dfrac{2 \sigma (\kappa^{2} - 1)}{[1 - \kappa \sigma + (\sigma - \kappa)t]^{2}}$. For $-1< \kappa <0$, then $\kappa^{2} - 1 <0$, then $q(t)$ decreases on $[\kappa + \varepsilon,1]$. Hence 
  $$q^{2}(t)_{\max} = \max\{q^{2}(\kappa + \varepsilon), q^{2}(1)\}.$$
  We have $q^{2}(1) = \left[-\dfrac{\sigma - 1}{\sigma + 1}\right]^{2}$, $q^{2}(\kappa + \varepsilon) = \left[\dfrac{1 - \kappa^{2} - \varepsilon(\sigma + \kappa)}{1 - \kappa^{2} + \varepsilon(\sigma - \kappa)}\right]^{2}$. For $\varepsilon$ is small enough, then $q^{2}(t)_{\max} = q^{2}(\kappa + \varepsilon)$.
  
  From above analysis, we can get there exists a constant $C_{\varepsilon}$ associated with $\varepsilon$, such that $r_{\Gamma}(x) \leq C_{\varepsilon}$, and for $t_{\Gamma}(x) = 1 - r_{\Gamma}(x)$, then $C_{\varepsilon} < t_{\Gamma}(x) <1$.
\end{proof}

 We can also get the continuity of $t_{\Gamma}(x)$.

 \begin{proposition}
   Suppose that $\Gamma = \{\rho(x)x;~x\in \bar{\Omega}\}$ is refractor from $\bar{\Omega}$ to $\bar{\Omega}^{\ast}$ and $E$ is the singular set of $\Gamma$, then $t_{\Gamma}(x)$ is continuous on $\bar{\Omega}\backslash E$.
   \label{prop4.2}
 \end{proposition}

 The following lemma and theorem are similar as those in Section 3.2, which are useful in proving the existence of the weak solution.

 \begin{lemma}
   Let $\Gamma_{k}$ and $\Gamma$ be refractors with defining functions $\rho_{k}(x)$ and $\rho(x)$, the corresponding fresnel coefficients are $t_{k}$ and $t$. Suppose that $\rho_{k}\rightarrow \rho$ pointwise in $\bar{\Omega}$ and there exist constants $C_{1},C_{2}>0$, such that $C_{1} \leq \rho_{k}(x) \leq C_{2}$ in $\bar{\Omega}$. Then for $y \notin E$, there exists subsequence $t_{k_{j}}(y)\rightarrow t(y)$ as $j\rightarrow \infty$, where $E$ is the union of singular points of refractors $\Gamma_{k}$ and $\Gamma$.
   \label{lem4.7}
 \end{lemma}

 \begin{theorem}
   Assume the hypotheses and notations of Lemma \ref{lem4.7} hold, and let $F \subseteq \bar{\Omega}^{\ast}$ be a compact set, set $F_{k} = \mathcal{T}_{\Gamma_{k}}(F)$. Then for all $y \notin E$, we have
  \begin{equation}\label{4.4}
    (a) \varlimsup_{k \rightarrow \infty} \chi_{F_{k}}(y)t_{k}(y) = t(y)\varlimsup_{k \rightarrow \infty} \chi_{F_{k}}(y);
  \end{equation}
  \begin{equation}\label{4.5}
    (b) \varliminf_{k \rightarrow \infty} \chi_{F_{k}}(y)t_{k}(y) = t(y)\varliminf_{k \rightarrow \infty} \chi_{F_{k}}(y).
  \end{equation}
   \label{thm4.1}
 \end{theorem}

\subsection{Weak solution of the far field refraction problem for the case $-1< \kappa <0$ with loss of energy}

In this subsection, we will give the definition of refractor measure and the weak solution of the far field refraction problem for the case $-1< \kappa <0$ with loss of energy. The definition of refractor measure is same as Definition \ref{def3.3}.

\begin{definition}
  Suppose $\Gamma$ is a refractor from $\bar{\Omega}$ to $\bar{\Omega}^{\ast}$, $f \in L^{1}(\bar{\Omega})$ and $\inf\limits_{\bar{\Omega}}f >0$. The refractor measure associated with $\Gamma$ and $f$ is defined by a set function on Borel subsets of $\bar{\Omega}^{\ast}$:
  \begin{equation}\label{4.6}
    G_{\Gamma}(F) := \int_{\mathcal{T}_{\Gamma}(F)}f(x)t_{\Gamma}(x)dx,
  \end{equation}
  where $dx$ is the surface measure on $S^{n-1}$.
  \label{def4.3}
\end{definition}

Based on the definition of refractor measure, we can define the weak solution of the far field refraction problem for the case $-1< \kappa <0$ with loss of energy.

\begin{definition}
  Suppose that $\mu$ is a Radon measure on the Borel subset of $\bar{\Omega}^{\ast}$ and $f \in L^{1}(\bar{\Omega})$, a refractor $\Gamma$ is a weak solution of the far field refraction problem for the case $-1 <\kappa <0$ with emitting illumination intensity $f(x)$ and prescribe refracted illumination intensity $\mu$ if for any Borel set $\omega \subseteq \bar{\Omega}^{\ast}$, there holds:
  \begin{equation}\label{4.7}
    G_{\Gamma}(\omega) = \int_{\mathcal{T}_{\Gamma}(\omega)}f(x)t_{\Gamma}(x)dx \geq \mu(\omega).
  \end{equation}
  \label{def4.4}
\end{definition}

\begin{remark}
    Similar to Remark \ref{rem3.4}, in order to ensure that light can be refracted into $\bar{\Omega}^{\ast}$, we use ``$\geq$'' in (\ref{4.7}).
  \label{rem4.3}
\end{remark}

\begin{remark}
  From Definition \ref{def4.4}, we can get that if $\Gamma$ is a weak solution of the refraction problem, then for any $c>0$, $c\Gamma$ is also a weak solution of the refraction problem, namely the weak solution is unique up to a multiplicative constant.
  \label{rem4.4}
\end{remark}

The existence of the weak solution is discussed in the following two subsections.

\subsection{Existence of the weak solution when $\mu$ is discrete measure}

In this subsection, we assume that $\mu$ equals finite sum of $\delta$-measures, hence all rays are refracted into finite directions. Based on this assumption, we establish the existence of the weak solution of the far field refraction problem for the case $-1 < \kappa <0$ with loss of energy when $\mu$ is discrete measure.

\begin{remark}
  Suppose that $m_{1},m_{2},\ldots, m_{l},~l \geq 2$ are discrete points in $\bar{\Omega}^{\ast}$, then for $\mathbf{b} = (b_{1},b_{2},\cdots ,b_{l}) \in \mathbb{R}^{l},~b_{i}>0$, the refractor is defined as
  \begin{equation}\label{4.8}
    \Gamma(\mathbf{b}) = \{\rho(x)x;~x\in \bar{\Omega},~\rho(x)=\min_{1\leq i \leq l}\frac{b_{i}}{1 - \kappa m_{i}\cdot x}\}.
  \end{equation}
  \label{rem4.5}
\end{remark}

Now we show the existence of the weak solution when $\mu$ equals the linear combination of the $\delta$-measures at $m_{1},m_{2},\ldots, m_{l}$. 

\begin{theorem}
  Suppose that $f \in L^{1}(\bar{\Omega})$ and $\inf\limits_{x \in \bar{\Omega}}f(x)>0$, $m_{1},m_{2},\ldots, m_{l},~l \geq 2$ are discrete points in $\bar{\Omega}^{\ast}$, $g_{1}, g_{2}, \ldots, g_{l}>0$. Let $\mu$ be the Borel measure defined on $\bar{\Omega}^{\ast}$ by $\mu = \sum\limits_{i = 1}^{l}g_{i}\delta_{m_{i}}(\omega)$, where $\omega \in \bar{\Omega}^{\ast}$ is Borel set. Also suppose that $$\displaystyle\int_{\bar{\Omega}}f(x)dx \geq \dfrac{1}{1 - C_{\varepsilon}} \mu(\bar{\Omega}^{\ast}),$$ where $C_{\varepsilon}$ is defined in Proposition \ref{prop4.1}. Then there exist $\mathbf{b}_{0} \in \mathbb{R}^{l}$ and refractor $\Gamma(\mathbf{b}_{0})$, such that
  $$\int_{\mathcal{T}_{\Gamma{(\mathbf{b_{0}})}(m_{i})}}f(x)t_{\Gamma(\mathbf{b}_{0})}(x)dx = g_{i}$$
  for $i = 2,\ldots,l$, and
  $$\int_{\mathcal{T}_{\Gamma{(\mathbf{b_{0}})}(m_{1})}}f(x)t_{\Gamma(\mathbf{b}_{0})}(x)dx > g_{1},$$
  namely the weak solution of the far field refraction problem for the case $-1 < \kappa <0$ with loss of energy exists.
  \label{thm4.2}
\end{theorem}

\begin{remark}
    Similar to Remark \ref{rem3.6}, in order to avoid this problem being overdetermined, we must have
    $$\int_{\mathcal{T}_{\Gamma{(\mathbf{b_{0}})}(m_{1})}}f(x)t_{\Gamma(\mathbf{b}_{0})}(x)dx > g_{1}.$$
    \label{rem4.6}
\end{remark}

To prove Theorem \ref{thm4.2}, we also need the following lemmas which are similar as Lemmas \ref{lem3.8}--\ref{lem3.10}.

\begin{lemma}
  Suppose that $f \in L^{1}(\bar{\Omega})$ and $\inf\limits_{x \in \bar{\Omega}}f(x)>0$, $m_{1},m_{2},\ldots, m_{l},~l \geq 2$ are discrete points in $\bar{\Omega}^{\ast}$, $g_{1}, g_{2}, \ldots, g_{l}>0$. Suppose that $W \subseteq \mathbb{R}^{l}$ is a set defined by $W := \{\mathbf{b} = (1,b_{2},\ldots, b_{l});~b_{i}>0\ {for} \ i=2, \ldots, l\}$, and for any $\mathbf{b} \in W$, $\Gamma(\mathbf{b})$ satisfies $G_{\Gamma(\mathbf{b})}(m_{i}) = \displaystyle\int_{\mathcal{T}_{\Gamma(\mathbf{b})}(m_{i})}f(x)t_{\Gamma(\mathbf{b})}(x)dx \leq g_{i}$ for $i = 2,\ldots,l$. Then we have:
  
   (a) $W \neq \emptyset$;
  
   (b) If $\mathbf{b} = (1,b_{2}, \ldots, b_{l}) \in W$, then $b_{i}>1 + \kappa$ for $i=2, \ldots, l$.
  \label{lem4.8}
\end{lemma}

\begin{proof}
  $(a)$ If for some $i \neq 1$, $E(m_{i},b_{i})$ supports $\Gamma$ at $\rho(x)x$, then we have
  $$\frac{b_{i}}{1 - \kappa} \leq \frac{b_{i}}{1 - \kappa m_{i} \cdot x} = \rho(x) \leq \frac{1}{1 - \kappa m_{1} \cdot x} \leq \frac{1}{1 - \kappa^{2}},$$
  so $b_{i} \leq \dfrac{1}{1 + \kappa}.$

  We claim that if for some $i \neq 1$, there holds $b_{i} > \dfrac{1}{1 + \kappa}$, then $\mathcal{T}_{\Gamma_{\mathbf{b}}(m_{i})} \subseteq E$, where $E$ is the singular point set of $\Gamma(\mathbf{b})$.

  Indeed, if $x \in \mathcal{T}_{\Gamma_{\mathbf{b}}(m_{i})}$, then there exists $b>0$, such that $E(m,b)$ supports $\Gamma$ at $\rho(x_{0})x_{0}$. Then we have
  $$\rho(x) = \min\limits_{1 \leq i \leq l}\frac{b_{i}}{1 - \kappa m_{i} \cdot x}, \quad \rho(x) \geq \frac{b}{1 - \kappa m_{i} \cdot x} \quad \text{and} \quad \rho(x_{0}) = \frac{b}{1 - \kappa m_{i} \cdot x_{0}},$$
  hence
  $$\frac{b}{1 - \kappa m_{i} \cdot x_{0}} = \rho(x_{0}) \leq \frac{b_{i}}{1 - \kappa m_{i} \cdot x_{0}},$$
  so we have $b < b_{i}$. If $b = b_{i}$, then $E(m_{i},b_{i})$ supports $\Gamma$, that is a contradiction, so $b < b_{i}$. Then for any $x \in \bar{\Omega}$, we have $\rho(x) \leq \dfrac{b}{1 - \kappa m_{i} \cdot x} < \dfrac{b_{i}}{1 - \kappa m_{i} \cdot x}$, so $\rho(x) = \min\limits_{j \neq 1}\dfrac{b_{j}}{1 - \kappa m_{j} \cdot x}$. Consequently, there exist $k \neq i$, such that
  $$\rho(x_{0}) = \frac{b_{k}}{1 - \kappa m_{k} \cdot x_{0}} \quad \text{and} \quad \rho(x) \leq \frac{b_{k}}{1 - \kappa m_{k} \cdot x}~\text{for all}~x\in \bar{\Omega}.$$
  So $x_{0} \in E$, then we have $\mathcal{T}_{\Gamma_{\mathbf{b}}(m_{i})} \subseteq E$.

  So we have
  $$G_{\Gamma(\mathbf{b})}(m_{i}) = \int_{\mathcal{T}_{\Gamma(\mathbf{b})}(m_{i})}f(x)t_{\Gamma(\mathbf{b})}(x)dx \leq  \int_{E}f(x)t_{\Gamma(\mathbf{b})}(x)dx < g_{i}.$$
  Take $\mathbf{b} = (1,b_{2}, \ldots, b_{l})$, such that $b_{i} > \dfrac{1}{1 + \kappa}$ for $2 \leq i \leq l$, then $\mathbf{b} \in W$, hence $W \neq \emptyset$.

  $(b)$ From Remark \ref{rem4.2}, we first claim that if $\mathbf{b} \in W$, then $g_{1} \leq G_{\Gamma(\mathbf{b})}(m_{1})$.

  Indeed, for we have
  \begin{align*}
     \sum_{i=1}^{l} G_{\Gamma(\mathbf{b})}(m_{i}) & =  \sum_{i=1}^{l} \int_{\mathcal{T}_{\Gamma(\mathbf{b})}(m_{i})}f(x)t_{\Gamma(\mathbf{b})}(x)dx \\
     & =  \int_{\bigcup\limits_{i=1}^{l}\mathcal{T}_{\Gamma(\mathbf{b})}(m_{i})}f(x)t_{\Gamma(\mathbf{b})}(x)dx  =  \int_{\bar{\Omega}}f(x)t_{\Gamma(\mathbf{b})}(x)dx\\
     & \geq (1 - C_{\varepsilon})\int_{\bar{\Omega}}f(x)dx \geq \mu(\bar{\Omega}^{\ast}) \\
     & = \sum_{i = 1}^{l} g_{i}.
  \end{align*}
  So we have
  $$g_{1} - G_{\Gamma(\mathbf{b})}(m_{1}) + \sum_{i=2}^{l}[g_{i} - G_{\Gamma(\mathbf{b})}(m_{i})] \leq 0. $$
  If $\mathbf{b} \in W$, then we have $g_{1} \leq G_{\Gamma(\mathbf{b})}(m_{1}).$

  Let $\Gamma(\mathbf{b}) = \{\rho(x)x;~x\in \bar{\Omega}\}$, we claim that there exists $\rho(x_{0})x_{0}$, such that $\rho(x_{0})x_{0} \in \Gamma(\mathbf{b}) \cap E(m_{1},1)$ and $\rho(x_{0})x_{0} \notin E(m_{i},b_{i})$ for all $i \geq 2$.

  Indeed, if not, we have $\mathcal{T}_{\Gamma(\mathbf{b})}(m_{1}) \subseteq E$, then
  $$G_{\Gamma(\mathbf{b})}(m_{1}) = \int_{\mathcal{T}_{\Gamma(\mathbf{b})}(m_{1})} f(x)t_{\Gamma(\mathbf{b})}(x)dx \leq \int_{E} f(x)t_{\Gamma(\mathbf{b})}(x)dx = 0$$
  for $\vert E \vert =0$. This is a contradiction with $g_{1} >0$, hence
  $$\rho(x_{0}) = \frac{1}{1 - \kappa m_{1} \cdot x_{0}} < \frac{b_{i}}{1 - \kappa m_{i} \cdot x_{0}},$$
  and thus we have
  $$b_{i} > \frac{1 - \kappa m_{i} \cdot x_{0}}{1 - \kappa m_{1} \cdot x_{0}} > 1 + \kappa.$$
\end{proof}

\begin{lemma}
  Let $\mathbf{b}_{k} = (b_{1}^{k},\ldots,b_{l}^{k})$ and $\mathbf{b}_{0} = (b_{1}^{0},\ldots,b_{l}^{0})$ with $\mathbf{b}_{k}\rightarrow \mathbf{b}_{0}$ in $\mathbb{R}^{l}$. Suppose that $\Gamma_{k} = \Gamma(\mathbf{b}_{k}) = \{\rho_{k}(x)x;~x\in \bar{\Omega}\}$, $\Gamma_{0} = \Gamma(\mathbf{b}_{0}) = \{\rho(x)x;~x\in \bar{\Omega}\}$, then $\rho_{k} \rightarrow \rho$ uniformly on $\bar{\Omega}$.
  \label{lem4.9}
\end{lemma}

\begin{proof}
  For $x_{0} \in \bar{\Omega}$, we have
  \begin{align*}
    \vert \rho_{k}(x_{0}) - \rho(x_{0}) \vert& = \vert \rho_{k}(x_{0}) - \frac{b_{i}}{1 - \kappa m_{i} \cdot x_{0}}  \vert \quad \text{for some}~i  \\
     & \leq  \vert \frac{b_{i}^{k}}{1 - \kappa m_{i} \cdot x_{0}} - \frac{b_{i}}{1 - \kappa m_{i} \cdot x_{0}} \vert \\
     & \leq \frac{\Vert \mathbf{b}^{k} - \mathbf{b} \Vert}{1 -\kappa^{2}},
  \end{align*}
  hence $\rho_{k} \rightarrow \rho$ uniformly on $\bar{\Omega}$.
\end{proof}

\begin{lemma}
  Let $\tau >0$, then $G_{\Gamma(\mathbf{b})}(m_{i}) = \displaystyle\int_{\mathcal{T}_{\Gamma(\mathbf{b})}(m_{i})}f(x)t_{\Gamma(\mathbf{b})}(x)dx$ is continuous on the region $R_{\tau} = \{(1,b_{2},\ldots, b_{l});~b_{i}> \tau,~i=2,\ldots,l\}$, for any $1 \leq i \leq l$.
  \label{lem4.10}
\end{lemma}

\begin{proof}
  Suppose that $\mathbf{b}_{k}$ is a sequence converges to $\mathbf{b}_{0}$ in $R_{\tau}$, and let $\Gamma(\mathbf{b}_{k}) = \{\rho_{k}(x)x;~x\in \bar{\Omega}\}$, $\Gamma(\mathbf{b}_{0}) = \{\rho_{0}(x)x;~x\in \bar{\Omega}\}$. Then from Lemma \ref{lem4.9}, $\rho_{k} \rightarrow \rho$ uniformly on $\bar{\Omega}$. Besides, for any $x\in \bar{\Omega}$ and $k \geq 1$, we have
  $$\rho_{k}(x) = \frac{b_{i}^{k}}{1 - \kappa m_{i} \cdot x} \geq \frac{\tau}{1 - \kappa m_{i} \cdot x} \geq \frac{\tau}{1 - \kappa}$$
  and
  $$\rho_{k}(x) = \min_{1 \leq i \leq l}\frac{b_{i}^{k}}{1 - \kappa m_{i} \cdot x} \leq \frac{1}{1 - \kappa m_{1} \cdot x} \leq\frac{1}{1 - \kappa^{2}}.$$
  for some $i \in \{1,2,\ldots,l\}$. Hence there exist $0 < a_{1} \leq a_{2}$, such that $ a_{1} \leq \rho_{k}(x) \leq  a_{2}$.

  Suppose that $G \in \bar{\Omega}^{\ast}$ is a neighborhood of $m_{i}$, such that $m_{j} \notin G$ for all $j \neq i$. If $x_{0} \in \mathcal{T}_{\Gamma({\mathbf{b}_{k}})}(G)$ and $x_{0} \notin E$, then there exist a unique $m \in G$ and $b >0$, such that
  $$\rho_{k}(x_{0}) = \frac{b}{1 - \kappa m \cdot x_{0}} \quad \text{and} \quad \rho_{k}(x) \leq \frac{b}{1 - \kappa m \cdot x}~\text{for all}~x \in \bar{\Omega}.$$
  From the definition of $\Gamma(\mathbf{b}_{k})$ in Lemma \ref{lem4.9}, we have $m = m_{j}$ for some $j = 1,2,\ldots,l$, hence we have $m = m_{j}$, then $\mathcal{T}_{\Gamma({\mathbf{b}_{k}})}(G) \subseteq \mathcal{T}_{\Gamma({\mathbf{b}_{k}})}(m_{i}) \cup E$. For $\vert E \vert = 0$, from Lemma \ref{lem4.6}, we have
  \begin{equation}\label{4.9}
    \begin{aligned}
    \int_{\mathcal{T}_{\Gamma({\mathbf{b}_{0}})}(G)}f(x)t_{\Gamma({\mathbf{b}_{0}})}(x)dx & \leq \int_{\varliminf\limits_{k \rightarrow \infty} \mathcal{T}_{\Gamma({\mathbf{b}_{k}})}(m_{i}) \cup E}f(x)t_{\Gamma({\mathbf{b}_{0}})}(x)dx \\
     & \leq \int_{\varliminf\limits_{k \rightarrow \infty} \mathcal{T}_{\Gamma({\mathbf{b}_{k}})}(m_{i})}f(x)t_{\Gamma({\mathbf{b}_{0}})}(x)dx + \int_{E}f(x)t_{\Gamma({\mathbf{b}_{0}})}(x)dx \\
     & = \int_{\bar{\Omega}}\chi_{\varliminf\limits_{k \rightarrow \infty} \mathcal{T}_{\Gamma({\mathbf{b}_{k}})}(m_{i})}f(x)t_{\Gamma({\mathbf{b}_{0}})}(x)dx.
  \end{aligned}
  \end{equation}
  Obviously, we have
  \begin{equation}\label{4.10}
   \chi_{\varliminf\limits_{k \rightarrow \infty} \mathcal{T}_{\Gamma({\mathbf{b}_{k}})}(m_{i})}(x) = \varliminf\limits_{k \rightarrow \infty} \chi_{\mathcal{T}_{\Gamma({\mathbf{b}_{k}})}(m_{i})}(x).
  \end{equation}
  Applying Theorem \ref{thm4.1}, (\ref{4.10}) and Fatou Lemma to (\ref{4.9}), we have
  \begin{equation}\label{4.11}
    \begin{aligned}
    \int_{\mathcal{T}_{\Gamma({\mathbf{b}_{0}})}(G)}f(x)t_{\Gamma({\mathbf{b}_{0}})}(x)dx & \leq \int_{\bar{\Omega}}\varliminf\limits_{k \rightarrow \infty} \chi_{\mathcal{T}_{\Gamma({\mathbf{b}_{k}})}(m_{i})}(x)t_{\Gamma({\mathbf{b}_{k}})}f(x)dx \\
    & \leq \varliminf\limits_{k \rightarrow \infty} \int_{\bar{\Omega}}\chi_{\mathcal{T}_{\Gamma({\mathbf{b}_{k}})}(m_{i})}(x)t_{\Gamma({\mathbf{b}_{k}})}f(x)dx \\
    & = \varliminf\limits_{k \rightarrow \infty} \int_{\mathcal{T}_{\Gamma({\mathbf{b}_{k}})}(m_{i})}t_{\Gamma({\mathbf{b}_{k}})}f(x)dx.
    \end{aligned}
  \end{equation}
  Besides, we also have
  \begin{equation}\label{4.12}
    \chi_{\varlimsup\limits_{k \rightarrow \infty} \mathcal{T}_{\Gamma({\mathbf{b}_{k}})}(m_{i})}(x) = \varlimsup\limits_{k \rightarrow \infty} \chi_{\mathcal{T}_{\Gamma({\mathbf{b}_{k}})}(m_{i})}(x).
  \end{equation}
  From inverse Fatou lemma, Lemma \ref{lem4.6}, Theorem \ref{thm4.1} and (\ref{4.12}), we have
  \begin{equation}\label{4.13}
    \begin{aligned}
     \varlimsup\limits_{k \rightarrow \infty} \int_{\mathcal{T}_{\Gamma({\mathbf{b}_{k}})}(m_{i})}t_{\Gamma({\mathbf{b}_{k}})}f(x)dx  & \leq \int_{\bar{\Omega}}\varlimsup\limits_{k \rightarrow \infty} \chi_{\mathcal{T}_{\Gamma({\mathbf{b}_{k}})}(m_{i})}(x)t_{\Gamma({\mathbf{b}_{k}})}f(x)dx  \\
     & = \int_{\bar{\Omega}}\chi_{\varlimsup\limits_{k \rightarrow \infty} \mathcal{T}_{\Gamma({\mathbf{b}_{k}})}(m_{i})}(x)f(x)t_{\Gamma(\mathbf{b}_{0})}(x)dx \\
     & = \int_{\varlimsup\limits_{k \rightarrow \infty} \mathcal{T}_{\Gamma({\mathbf{b}_{k}})}(m_{i})}f(x)t_{\Gamma(\mathbf{b}_{0})}(x)dx \\
     & \leq \int_{\mathcal{T}_{\Gamma({\mathbf{b}_{0}})}(G)}f(x)t_{\Gamma(\mathbf{b}_{0})}(x)dx.
    \end{aligned}
  \end{equation}
  Combining (\ref{4.11}) with (\ref{4.13}), we get $G_{\Gamma(\mathbf{b})}(m_{i})$ is continuous on the region $R_{\tau}$.
\end{proof}

Based on the above lemmas, we can prove the existence of the weak solution when $\mu$ is discrete measure.

\begin{proof}[Proof of Theorem \ref{thm4.2}]
  Fixed $\overline{\mathbf{b}} = (1,\overline{b_{2}},\ldots,\overline{b_{l}})$, consider the set $\overline{W} = \{\mathbf{b}_{i} = (1,b_{2},\ldots,b_{l});~
    b_{i} \leq \overline{b_{i}},~ i=2,\ldots,l\}$, from Lemma \ref{lem4.8} and Lemma \ref{lem4.10}, $W$ is a compact set. Define a mapping
  $$d:\overline{W}\rightarrow \mathbb{R};~\mathbf{b} \mapsto \sum_{i=1}^{l}b_{i}.$$
  Let $\mathbf{b}^{\ast} = \arg\min\limits_{\mathbf{b}\in \overline{W}}d(\mathbf{b})$, for the compactness of $\overline{W}$, then we know $d$ is a continuous mapping hence $\mathbf{b}^{\ast}$ exists.

  Taking $\mathbf{b}_{0} = \mathbf{b}^{\ast}$, we first prove that $\displaystyle\int_{\mathcal{T}_{\Gamma(\mathbf{b}_{0})}(m_{i})}f(x)t_{\Gamma(\mathbf{b}_{0})}(x)dx = g_{i}$ for $i = 2,\ldots,l$.

  Indeed, if not, we may assume that $\displaystyle\int_{\mathcal{T}_{\Gamma(\mathbf{b}_{0})}(m_{2})}f(x)t_{\Gamma(\mathbf{b}_{0})}(x)dx < g_{2}.$ Taking $\xi >1$ and let $\mathbf{b}_{\xi} = (1,\xi b_{2}^{\ast},\ldots, b_{l}^{\ast})$. If $x_{0} \in \mathcal{T}_{\Gamma(\mathbf{b}_{\xi})}(m_{i}) \setminus E_{\xi}^{\ast}$, where $E_{\xi}^{\ast}$ is the singular point set of $\mathbf{b}_{\xi}$, then we have
  $$\rho(x_{0}) = \frac{b_{i}^{\ast}}{1 - \kappa m_{i} \cdot x_{0}} \quad \text{and} \quad \rho(x) \leq \frac{b_{i}^{\ast}}{1 - \kappa m_{i} \cdot x}~\text{for all}~x\in \bar{\Omega},$$
  hence $x_{0} \in \mathcal{T}_{\Gamma(\mathbf{b}^{\ast})}(m_{i})$, then $\mathcal{T}_{\Gamma(\mathbf{b}_{\xi}^{\ast})}(m_{i}) \setminus E_{\xi}^{\ast} \subseteq \mathcal{T}_{\Gamma(\mathbf{b}^{\ast})}(m_{i})$. So we have
  $$\int_{\mathcal{T}_{\Gamma(\mathbf{b}_{\xi}^{\ast})}(m_{i})}f(x)t_{\Gamma(\mathbf{b}_{\xi}^{\ast})}dx = \int_{\mathcal{T}_{\Gamma(\mathbf{b}_{\xi}^{\ast})}(m_{i})}f(x)t_{\Gamma(\mathbf{b}^{\ast})}dx \leq \int_{\mathcal{T}_{\Gamma(\mathbf{b}^{\ast})}(m_{i})}f(x)t_{\Gamma(\mathbf{b}^{\ast})}dx.$$
  Let $\xi \rightarrow 1$, then from Lemma \ref{lem4.10}, we have $G_{\Gamma_{\mathbf{b}_{\xi}^{\ast}}} < g_{2}$, hence $b_{\xi}^{\ast} \in W$, this is a contradiction with $d(\mathbf{b}_{\xi}^{\ast}) \geq d(\mathbf{b}_{\xi})$.
  
  The proof of $\displaystyle\int_{\mathcal{T}_{\Gamma(\mathbf{b}_{0})}(m_{1})}f(x)t_{\Gamma(\mathbf{b}_{0})}(x)dx > g_{1}$ is same as which in Theorem \ref{thm3.3}. 
     \end{proof}

\subsection{Existence of the weak solution when $\mu$ is a finite Radon measure}

In this subsection, we discuss the existence of the weak solution of the far field refraction problem for the case $-1 < \kappa <0$ with loss of energy when $\mu$ is a finite Radon measure. We use the similar method as Section 3.5 to prove the following theorem.

\begin{theorem}
  Suppose that $f$ is integrable on $\bar{\Omega}$ and $\inf\limits_{x\in \bar{\Omega}}f(x)>0$. Let $\mu$ be a Radon measure on $\bar{\Omega}^{\ast}$ and
  \begin{equation}\label{4.14}
    \int_{\bar{\Omega}}f(x)dx \geq \frac{1}{1 - C_{\varepsilon}}\mu(\bar{\Omega}^{\ast}),
  \end{equation}
  where $C_{\varepsilon}$ is defined in Proposition \ref{prop4.1}. Then there exists a refractor $\Gamma$, such that for any Borel subset $\omega \subseteq \bar{\Omega}^{\ast}$, we have
  $$\mu(\omega) \leq \int_{\mathcal{T}_{\Gamma}(\omega)}f(x)t_{\Gamma}(x)dx,$$
  that is, there exists a weak solution of the refraction problem for the case $-1 <\kappa < 0$ with emitting illumination intensity $f$ and prescribed refracted intensity $\mu$.
  \label{thm4.3}
\end{theorem}

\begin{proof}
  Let $\iota$ be an integer, $\iota \geq 2$. Segmenting $\bar{\Omega}^{\ast}$ into finite disjoint subsets $\omega_{1}^{\iota}, \omega_{2}^{\iota}, \ldots,\omega_{l_{\iota}}^{\iota}$, such that $diam(\omega_{i}^{\iota}) \leq \dfrac{1}{\iota}$ for $i = 1,2,\ldots,l_{\iota}$. Take $m_{i}^{\iota} \in \omega_{i}^{\iota}$ and consider the measure $\mu_{\iota} := \sum\limits_{i = 1}^{l _{\iota}}\mu(\omega_{i}^{\iota}) \delta_{m_{i}^{\iota}}$ defined on $\bar{\Omega}^{\ast}$. From the proof of Theorem \ref{thm3.4}, we know that $\mu_{\iota} \rightarrow \mu$ weakly as $\iota \rightarrow \infty$.

  From (\ref{4.14}), we have $\mu_{\iota}(\bar{\Omega}^{\ast}) = \mu(\bar{\Omega}^{\ast}) \leq (1 - C_{\varepsilon})\displaystyle\int_{\bar{\Omega}}f(x)dx$, then from Theorem \ref{thm4.2}, there exists a refractor $\Gamma_{\iota} = \{\rho_{\iota}(x)x;~\rho_{\iota}(x) = \min\limits_{1 \leq i \leq l_{\iota}}\dfrac{b_{i}}{1 - \kappa m_{i}^{\iota} \cdot x}\}$, such that $\mu_{\iota}(\omega) \leq \displaystyle\int_{\mathcal{T}_{\Gamma_{\iota}}(\omega)}f(x)t_{\Gamma_{\iota}}(x)dx$. Normalized $\Gamma_{\iota}$, such that $\inf\limits_{x \in \bar{\Omega}}\rho_{\iota}(x) = 1$, then from Lemma \ref{lem4.2}, there exists a constant $C>0$, such that $\sup\limits_{x \in \bar{\Omega}}\rho_{\iota}(x) \leq C$ for all $\iota \geq 1$.

  Besides, if $x_{0}, x_{1} \in \bar{\Omega}$ and $E(m_{0},b_{0})$ supports $\Gamma_{\iota}$ at $\rho_{\iota}(x_{0})x_{0}$, then for $x_{1} \in \bar{\Omega}$, we have
  \begin{align*}
    \vert \rho_{\iota}(x_{1}) - \rho_{\iota}(x_{0}) \vert & \leq \vert \frac{b_{0}}{1 - \kappa m_{0} \cdot x_{1}} - \frac{b_{0}}{1 - \kappa m_{0} \cdot x_{0}} \vert\\
      & \leq \frac{-\kappa b_{0}}{(1 - \kappa m_{0} \cdot x_{1})(1 - \kappa m_{0} \cdot x_{0})} \Vert x_{1} - x_{0} \Vert \\
      & \leq \frac{-\kappa}{1 - \kappa^{2}}\frac{b_{0}}{1 - \kappa^{2}} \Vert x_{1} - x_{0} \Vert\\
      & \leq \frac{C}{1 - \kappa^{2}} \Vert x_{1} - x_{0} \Vert.
  \end{align*}
  Exchanging the roles of $x_{1}$ and $x_{0}$, we have
  $$\vert \rho_{\iota}(x_{0}) - \rho_{\iota}(x_{1}) \vert \leq \frac{C}{1 - \kappa^{2}} \Vert x_{0} - x_{1} \Vert,$$
  hence $\{\rho_{\iota}(x);~\iota \geq 1\}$ is a family of bounded uniformly and equicontinuous functions. Then from Arezlà-Ascoli Theorem, $\rho_{\iota}(x) \rightarrow \rho(x)$ uniformly as $\iota \rightarrow \infty$ for all $x \in \bar{\Omega}$. Then from Lemma \ref{lem4.6} $(a)$, $\Gamma = \{\rho(x)x;~x \in \bar{\Omega}\}$ is a refractor. 
  
  Similar as the proof of Theorem \ref{thm3.4}, we also have $G_{\Gamma{\iota}} := \displaystyle \int_{\mathcal{T}_{\Gamma_{\iota}}(\omega)}f(x)t_{\Gamma_{\iota}}(x)dx \rightarrow G_{\Gamma} := \displaystyle \int_{\mathcal{T}_{\Gamma}(\omega)}f(x)t_{\Gamma}(x)dx$ weakly as $\iota \rightarrow \infty$, hence the weak solution of the problem exists.
\end{proof}

\section{The inequality for the problem}\label{Section 5}

\sloppy{}

In this section, we derive inequality involving a Monge-Amp\`ere type operator satisfied by $\rho$. We first recall the Jacobian equation given in \cite{GM13}.

Let $X = (x,x_{n})$ be a point in the sphere $S^{n-1}$, where $x = (x_{1},\ldots,x_{n-1})$. Let $\Gamma = \{\rho(X)X;~X\in \bar{\Omega}\}$ be a weak solution of the refractor problem from $\bar{\Omega}$ to $\bar{\Omega}^{\ast}$ with emitting illumination intensity $f$ and prescribed refracted illumination intensity $g$. Assume that $\bar{\Omega}$ is a subset of upper unit sphere $S^{n-1}_{+} = S^{n-1} \cap \{x_{n} >0\}$, then $\bar{\Omega}$ can be identified by its orthogonal projection $\mathcal{V} = \{x = (x_{1},\ldots,x_{n-1});~(x,\sqrt{1 - \vert x \vert^{2}}) \in \bar{\Omega}\}$. Suppose that $\rho$ is a function of $x$ with $x \in \mathcal{V}$. For convenience, we may assume that $\rho \in C^{2}(\bar{\Omega})$.

Let $Y$ be the refracted direction of the ray $X$ by the surface $\rho(X)X$, then from Snell law (\ref{2.3}), we have
\begin{equation}\label{5.1}
  Y = \frac{1}{\kappa}(X - \Phi(x \cdot \nu)\nu),
\end{equation}
where $\Phi$ is given by (\ref{2.5}) and $\nu$ is the outward unit normal to the refractor at $\rho(X)X$.

Define a map $T:\mathcal{V}\rightarrow \bar{\Omega}^{\ast}:X \mapsto Y$, then the Jacobian matrix of $T$ is given by
\begin{equation*}
  \left(
  \begin{array}{cccc}
    \partial_{1} y_{1} & \ldots & \partial_{n-1} y_{1} & y_{1} \\
    \partial_{1} y_{2} & \ldots & \partial_{n-1} y_{2} & y_{2} \\
    \vdots & \ddots & \vdots & \vdots \\
    \partial_{1} y_{n-1} & \ldots & \partial_{n-1} y_{n-1} & y_{n-1} \\
    \partial_{1} y_{n} & \ldots & \partial_{n-1} y_{n} & y_{n}
  \end{array}
  \right).
\end{equation*}
Then we have
\begin{equation}\label{5.2}
  \det J = \frac{1}{y_{n}} \det Dy,
\end{equation}
where
\begin{equation*}
  Dy =
  \left(
  \begin{array}{ccc}
    \partial_{1} y_{1} & \ldots & \partial_{n-1} y_{1}  \\
    \partial_{1} y_{2} & \ldots & \partial_{n-1} y_{2} \\
    \vdots & \ddots & \vdots  \\
    \partial_{1} y_{n-1} & \ldots & \partial_{n-1} y_{n-1}
  \end{array}
  \right).
\end{equation*}

Let $dS_{\bar{\Omega}^{\ast}}$ be the area elements corresponding to $\bar{\Omega}^{\ast}$, $dS_{\mathcal{V}}$ be the volume element corresponding to $\mathcal{V}$, then
\begin{equation}\label{5.3}
  \vert \det J \vert = \frac{dS_{\bar{\Omega}^{\ast}}}{dS_{\mathcal{V}}}.
\end{equation}
Similarly, if $dS_{\bar{\Omega}}$ denotes the area elements corresponding to $\bar{\Omega}$, then
\begin{equation}\label{5.4}
  \frac{dS_{\bar{\Omega}}}{dS_{\mathcal{V}}} = \frac{1}{\sqrt{1 - \vert x \vert^{2}}}.
\end{equation}

Suppose that $x_{0} \notin E$, where $E$ is the singular point set of $\Gamma$, and $m_{0} = \mathcal{T}_{\Gamma}(x_{0}) = T(x_{0})$. 
Combing~(\ref{5.3}), (\ref{5.4}) and Lebesgue differentiation theorem with the following energy condition
$$\int_{\mathcal{T}_{\Gamma}(\omega)}f(x)t_{\Gamma}(x)dx \geq \int_{\omega}g(m)dm,$$
we obtain the Jacobian equation
\begin{equation}\label{5.5}
   \vert \det J \vert = \frac{dS_{\bar{\Omega}^{\ast}}}{dS_{\mathcal{V}}} \leq \frac{f(x)t_{\Gamma}(x)}{\sqrt{1 - \vert x \vert^{2}}g(T(x))}.
\end{equation}

Next, we derive the inequality involving a Monge-Amp\`ere type operator satisfied by $\rho$.

\begin{theorem}
  Suppose a refractor $\Gamma$ is defined by $\rho$, and $\rho$ is the weak solution to the refractor problem in negative refractive index material with loss of energy with emitting illumination intensity $f \in L^{1}(\bar{\Omega})$ and prescribed refracted illumination intensity $g \in L^{1}(\bar{\Omega}^{\ast})$, then we have
  \begin{equation}\label{5.6}
    \vert \det (D^{2}\rho + C^{-1}B) \vert \leq \frac{f(x)t_{\Gamma}(x)\vert \kappa \vert^{n - 2} \omega}{g(T(x))h^{n - 1}\left(1 - h^{-1}\left(\dfrac{\rho}{1 - \vert x \vert^{2}}x - D\rho\right)\cdot D_{p}h\right)},
  \end{equation}
  where $C^{-1}$ is given in (\ref{5.25}), $B$ is given by (\ref{5.23}), $h$ and $\omega$ are defined in (\ref{5.14}) and (\ref{5.15}) correspondingly.
  \label{thm5.1}
\end{theorem}

\begin{proof}
  In order to prove (\ref{5.6}), we first need to derive the unit outer normal $\nu$ to the surface $\Gamma$ at $\rho(X)X$.

  Write $\nu = (\nu',\nu_{n})$, for $\partial_{x_{k}}((x,x_{n})\rho(x))$ are tangential to the graph of the refractor $\Gamma$ for $k = 1,2,\ldots,n-1$, then we have
  $$\partial_{x_{k}}((x,x_{n})\rho(x))\cdot \nu = 0$$
  for $k = 1,2,\ldots,n-1$. Hence we have
  \begin{equation}\label{5.7}
    \rho \sum_{i = 1}^{n-1}\delta_{ik}\nu_{i} + \partial_{x_{k}}\rho \sum_{i = 1}^{n-1}x_{i}\nu_{i} = \left(\rho \frac{x_{k}}{\sqrt{1 - \vert x \vert^{2}}} - \sqrt{1 - \vert x \vert^{2}}\partial_{x_{k}}\rho\right)
  \end{equation}
  for $k = 1,2,\ldots,n-1$. Using the tensor product, (\ref{5.7}) can be written in matrix form
  \begin{equation}\label{5.8}
    (\rho I + D\rho \otimes x)(\nu')^{T} = \left(\rho \frac{x^{T}}{\sqrt{1 - \vert x \vert^{2}}} - \sqrt{1 - \vert x \vert^{2}}(D\rho)^{T}\right)\nu_{n}.
  \end{equation}
  According to Shermann-Morrison formula, we have
  \begin{equation*}
    (\rho I + D\rho \otimes x)^{-1} = \rho^{-1}\left(I - \frac{D\rho \otimes x}{\rho + D\rho \cdot x}\right).
  \end{equation*}
  Notice that for any row vectors $a,b$ and $c$, we have $(a \otimes b)c^{T} = (b \cdot c)a^{T}$, then we have
  \begin{align*}
    (\nu')^{T} & = \rho^{-1}\left(I - \frac{D\rho \otimes x}{\rho + D\rho \cdot x}\right)\left(\rho \frac{x^{T}}{\sqrt{1 - \vert x \vert^{2}}} - \sqrt{1 - \vert x \vert^{2}}(D\rho)^{T}\right)\nu_{n} \\
     & = \rho^{-1}\left(\frac{\rho}{\sqrt{1 - \vert x \vert^{2}}}x^{T} - \sqrt{1 - \vert x \vert^{2}}(D\rho)^{T} - \frac{\rho}{\sqrt{1 - \vert x \vert^{2}}(\rho + D\rho \cdot x)}(D\rho \otimes x)x^{T}\right. \\
     & \quad \left. + \frac{\sqrt{1 - \vert x \vert^{2}}}{\rho + D\rho \cdot x}(D\rho \otimes x)(D\rho)^{T}\right)\nu_{n} \\
     & = \rho^{-1}\left(\frac{\rho}{\sqrt{1 - \vert x \vert^{2}}}x^{T} - (\sqrt{1 - \vert x \vert^{2}} + \frac{\vert x \vert^{2}\rho}{\sqrt{1 - \vert x \vert^{2}}(\rho + D\rho \cdot x)}\right. \\
     & \quad \left. - \frac{\sqrt{1 - \vert x \vert^{2}}}{\rho + D\rho \cdot x}(x \cdot D\rho))(D\rho)^{T}\right)\nu_{n} \\
     & = \frac{1}{\sqrt{1 - \vert x \vert^{2}}}\left(x^{T} - \frac{1}{\rho + D\rho \cdot x}(D\rho)^{T}\right)\nu_{n}.
  \end{align*}
  Hence we have
  \begin{equation}\label{5.9}
    \nu = \left(\frac{1}{\sqrt{1 - \vert x \vert^{2}}}\left(x - \frac{1}{\rho + D\rho \cdot x}D\rho\right),1\right)\nu_{n}.
  \end{equation}
  From (\ref{5.9}), we have
  \begin{equation}\label{5.10}
    X\cdot \nu = \frac{1}{\sqrt{1 - \vert x \vert^{2}}}\left(\frac{\rho}{\rho + D\rho \cdot x}\right)\nu_{n}.
  \end{equation}
  Since $\nu$ is unit outer normal to $\Gamma$ at $\rho(X)X$, then we have $X\cdot \nu \geq 0$ and $\vert \nu' \vert^{2} + \nu_{n}^{2} = 1$, then from (\ref{5.9}), we have
  \begin{equation}\label{5.11}
    \left(\frac{\rho^{2} - (x \cdot D\rho)^{2} + \vert D\rho \vert^{2}}{(1 - \vert x \vert^{2})(\rho + D\rho \cdot x)^{2}}\right)\nu_{n}^{2} = 1.
  \end{equation}
  According to (\ref{5.11}), we obtain
  \begin{equation}\label{5.12}
    \nu_{n} = \pm \vert \rho + D\rho \cdot x \vert\sqrt{\frac{1-\vert x \vert^{2}}{\rho^{2} - (x \cdot D\rho)^{2} + \vert D\rho \vert^{2}}}.
  \end{equation}
  Then from (\ref{5.9}), we get
  \begin{equation}\label{5.13}
    \nu = \pm \frac{\vert \rho + D\rho \cdot x \vert}{\rho + D\rho \cdot x}\left(\frac{-D\rho +\rho + D\rho \cdot x}{\sqrt{\rho^{2} - (x \cdot D\rho)^{2} + \vert D\rho \vert^{2}}}x,\frac{\sqrt{1 - \vert x \vert^{2}}(\rho + D\rho \cdot x)}{\sqrt{\rho^{2} - (x \cdot D\rho)^{2} + \vert D\rho \vert^{2}}}\right).
  \end{equation}
  Besides, from (\ref{5.13}), (\ref{5.10}) can be written as
  \begin{equation*}
    X\cdot \nu = \pm \frac{\vert \rho + D\rho \cdot x \vert}{\rho + D\rho \cdot x}\frac{\rho}{\sqrt{\rho^{2} - (x \cdot D\rho)^{2} + \vert D\rho \vert^{2}}}.
  \end{equation*}

  Now we can derive the inequality involving a Monge-Amp\`ere type operator satisfied by $\rho$. For simplicity, we introduce two functions:
  \begin{equation}\label{5.14}
    h(x,z,p) = \frac{\Phi\left(\dfrac{z}{\sqrt{z^{2}+\vert p \vert^{2}-(p \cdot x)^{2}}}\right)}{\sqrt{z^{2}+\vert p \vert^{2}-(p \cdot x)^{2}}}
  \end{equation}
  and
  \begin{equation}\label{5.15}
   \omega(x,z,p) = 1 - h(x,z,p)(z + p \cdot x),
  \end{equation}
  where $\Phi$ is defined in (\ref{2.5}). Then from (\ref{5.1}), we have
  \begin{equation}\label{5.16}
    y_{i} = \frac{1}{\kappa}\left[\omega(x,\rho(x),D\rho(x))x_{i} + h(x,\rho(x),D\rho(x))\rho_{x_{i}}\right], \quad 1 \leq i \leq n-1,
  \end{equation}
  and
  \begin{equation}\label{5.17}
    y_{n} = \frac{1}{\kappa} \omega(x,\rho(x),D\rho(x)) \sqrt{1 - \vert x \vert^{2}}.
  \end{equation}
  For $1 \leq i,j \leq n-1$, differentiating $y_{i}$ with respect to $x_{j}$, we have
  \begin{equation}\label{5.18}
    \begin{aligned}
    \partial_{j}y_{i} & = \frac{1}{\kappa}\left[\omega \delta_{ij} + x_{i}(\omega_{x_{j}} + \omega_{z} \rho_{x_{j}} + \sum_{k = 1}^{n-1}\omega_{p_{k}} \rho_{x_{k}x_{j}})\right.  
      \left. + h \rho_{x_{i}x_{j}} + \rho_{x_{i}}(h_{x_{j}} + h_{z} \rho_{x_{j}} + \sum_{k=1}^{n-1} h_{p_{k}} \rho_{x_{k}x_{j}})\right].
    \end{aligned}
  \end{equation}
  For $x, D\rho, D_{x}\omega, D_{p}\omega, D_{z}\omega, D_{z}\omega, D_{x}h$ and $D_{p}h$ are row vectors, then (\ref{5.18}) can be written in matrix form
  \begin{equation}\label{5.19}
    \begin{aligned}
    Dy & = \frac{1}{\kappa}[\omega I + x \otimes D_{x}\omega + \omega_{z}x \otimes D\rho + D\rho \otimes D_{x}h  \\
    & \quad + h_{z}D\rho \otimes D\rho + (x \otimes D_{p}\omega)D^{2}\rho + hD^{2}\rho + (D\rho \otimes D_{p}h)D^{2}\rho].
    \end{aligned}
  \end{equation}
  Let
  \begin{equation}\label{5.20}
    B(x) = \omega I + x \otimes D_{x}\omega + \omega_{z}x \otimes D\rho + D\rho \otimes D_{x}h + h_{z}D\rho \otimes D\rho,
  \end{equation}
  and
  \begin{equation}\label{5.21}
    C(x) = x \otimes D_{p}\omega + hI + D\rho \otimes D_{p}h.
  \end{equation}
  Then (\ref{5.19}) can be written as
  \begin{equation}\label{5.22}
    Dy = \frac{1}{\kappa}[B(x) + C(x)D^{2}\rho].
  \end{equation}
  From (\ref{5.15}), we have $D_{x}\omega = -D_{x}h -hp$, $\omega_{z} = -h_{z}(z + p \cdot x) - h$ and $D_{p}\omega = -D_{p}h(z + p \cdot x) - hx$, then we have
  \begin{equation}\label{5.23}
    \begin{aligned}
    B(x) & = [1 - (\rho + D\rho \cdot x)h]I - [(\rho + D\rho \cdot x)x - D\rho]\otimes D_{x}h  \\
    & \quad - x\otimes [(2h + h_{z}((\rho + D\rho \cdot x)))D\rho] + h_{z}D\rho \otimes D\rho,
    \end{aligned}
  \end{equation}
  and
  \begin{equation}\label{5.24}
    \begin{aligned}
    C(x) & = [(-(\rho + D\rho \cdot x)x + D\rho) \otimes D_{p}h] - h[(x \otimes x) - I] \\
    & = h[(h^{-1}(-(\rho + D\rho \cdot x)x + D\rho)\otimes D_{p}h) + (((-x) \otimes x) + I)] \\
    & = h(M_{1} + M_{2}) \\
    & = h M_{2}(I + M_{2}^{-1}M_{1}),
    \end{aligned}
  \end{equation}
  where
  \begin{equation*}
    M_{1} = -h^{-1}((\rho + D\rho \cdot x)x - D\rho) \otimes D_{p}h \quad \text{and} \quad M_{2} = ((-x)\otimes x) + I.
  \end{equation*}
  From Shermann-Morrison formula, we have
 $
    M_{2}^{-1} = I + \frac{x \otimes x}{1 - \vert x \vert^{2}},
 $
  and
 $
    C^{-1} = \frac{1}{h}(I + M_{2}^{-1}M_{1})M_{2}^{-1}.
  $
  We may assume that $v = h^{-1}[(\rho + D\rho \cdot x)x - D\rho]$, then $I + M_{2}^{-1}M_{1} = I + (-M_{2}^{-1}v^{T})D_{p}h$, then we have
  \begin{equation*}
    N := (I + M_{2}^{-1}M_{1})^{-1} = I + \frac{(-M_{2}^{-1}v^{T})D_{p}h}{1 - (-M_{2}^{-1}v^{T})^{T}\cdot D_{p}h},
  \end{equation*}
  hence
  \begin{equation}\label{5.25}
    C^{-1} = \frac{1}{h}N\left(I + \frac{x \otimes x}{1 - \vert x \vert^{2}}\right).
  \end{equation}
  Now we calculate the matrix $N$ accurately. We have
  \begin{align*}
    M_{2}^{-1}v^{T} & = v^{T} + \frac{1}{1 - \vert x \vert^{2}} x^{T}xv^{T} \\
      & = h^{-1}\left[(\rho + D\rho \cdot x)x^{T} - (D\rho)^{T} + \frac{1}{1 - \vert x \vert^{2}}(\vert x \vert^{2}(\rho + D\rho \cdot x)x^{T} - (D\rho \cdot x)x^{T})\right] \\
      & = h^{-1}\left[\frac{\rho}{1 - \vert x \vert^{2}}x^{T} - (D\rho)^{T}\right].
  \end{align*}
  Hence, we have
  \begin{equation}\label{5.26}
    N = I + \frac{h^{-1}\left(\dfrac{\rho}{1 - \vert x \vert^{2}}x - D\rho\right)\otimes D_{p}h}{1 - h^{-1}\left(\dfrac{\rho}{1 - \vert x \vert^{2}}x - D\rho\right)\cdot D_{p}h},
  \end{equation}
  again from Shermann-Morrison formula, we obtain
    \begin{equation}\label{5.27}
    \det N = \frac{1}{1 - h^{-1}\left(\dfrac{\rho}{1 - \vert x \vert^{2}}x - D\rho\right)\cdot D_{p}h}.
  \end{equation}
  Combining (\ref{5.25}) with (\ref{5.27}), we have
    \begin{equation}\label{5.28}
    \det C = \frac{1}{\det C^{-1}} = h^{n-1}(1 - \vert x \vert^{2})\left[1 - h^{-1}\left(\frac{\rho}{1 - \vert x \vert^{2}}x - D\rho\right)\cdot D_{p}h\right].
  \end{equation}
  Substituting (\ref{5.26}) into (\ref{5.25}), we have
    \begin{align*}
    C^{-1} & = \frac{1}{h}\left[I + \frac{x\otimes x}{1 - \vert x \vert^{2}} + \frac{(M_{2}^{-1}v^{T})D_{p}h}{1 - (M_{2}^{-1}v^{T})^{T}\cdot D_{p}h} + \left(\frac{(M_{2}^{-1}v^{T})D_{p}h}{1 - (M_{2}^{-1}v^{T})^{T}\cdot D_{p}h}\right)\left(\frac{x \otimes x}{1 - \vert x \vert^{2}}\right)\right] \\
      & = \frac{1}{h}\left[I + \frac{x\otimes x}{1 - \vert x \vert^{2}} + \frac{1}{1 - (M_{2}^{-1}v^{T})^{T}\cdot D_{p}h}(M_{2}^{-1}v^{T})\left(D_{p}h + \frac{x \cdot D_{p}h}{1 - \vert x \vert^{2}}x\right)\right] \\
      & = \frac{1}{h}\left[I + \frac{x\otimes x}{1 - \vert x \vert^{2}}\right.  \\
      & \quad \left. + \frac{1}{1 - (M_{2}^{-1}v^{T})^{T}\cdot D_{p}h}\left(h^{-1}\left(\frac{\rho}{1 - \vert x \vert^{2}}x^{T} - (D\rho)^{T}\right)\left(D_{p}h + \frac{x \cdot D_{p}h}{1 - \vert x \vert^{2}}x\right)\right)\right],
  \end{align*}
  where we have used the fact that for the row vectors $a,b,c,d$, $(a \otimes b)(c \otimes d) = (b \cdot c)(a \otimes d)$. Denoting that
  \begin{equation*}
    A = \frac{1}{1 - (M_{2}^{-1}v^{T})^{T}\cdot D_{p}h}\left(h^{-1}\left(\frac{\rho}{1 - \vert x \vert^{2}}x^{T} - (D\rho)^{T}\right)\left(D_{p}h + \frac{x \cdot D_{p}h}{1 - \vert x \vert^{2}}x\right)\right),
  \end{equation*}
  then from
  \begin{equation*}
    (M_{2}^{-1}v^{T})^{T}\cdot D_{p}h = h^{-1}\left[\frac{\rho}{1 - \vert x \vert^{2}}(x \cdot D_{p}h) - D\rho \cdot D_{p}h\right],
  \end{equation*}
  we have
  \begin{equation*}
    A = \frac{1}{h - \left[\dfrac{\rho}{1 - \vert x \vert^{2}}(x \cdot D_{p}h) - D\rho \cdot D_{p}h\right]}\left(\frac{\rho}{1 - \vert x \vert^{2}}x - D\rho\right)\otimes \left(D_{p}h + \frac{x \cdot D_{p}h}{1 - \vert x \vert^{2}}x\right),
  \end{equation*}
  so (\ref{5.25}) can be written as
  \begin{equation*}
    C^{-1} = \frac{1}{h}\left[I + \frac{x \otimes x}{1 - \vert x \vert^{2}} + A\right].
  \end{equation*}

  From (\ref{5.22}), we have
  \begin{equation*}
    Dy = \frac{1}{\kappa}C(C^{-1}B + D^{2}\rho),
  \end{equation*}
  hence
  \begin{equation}\label{5.29}
    \det Dy = \frac{1}{\kappa^{n-1}} \det C \det (C^{-1}B + D^{2}\rho).
  \end{equation}
  Combining (\ref{5.2}), (\ref{5.5}), (\ref{5.17}) and (\ref{5.25}), we have
  \begin{equation*}
    \vert \frac{1}{\kappa^{n-1}} \det C \det (C^{-1}B + D^{2}\rho) \vert = \frac{f(x)t_{\Gamma}(x)\omega}{\kappa g(T(x))},
  \end{equation*}
  then from (\ref{5.28}), we finally obtain
  \begin{equation*}
    \vert \det (D^{2}\rho + C^{-1}B) \vert \leq \frac{f(x)t_{\Gamma}(x)\vert \kappa \vert^{n - 2} \omega}{g(T(x))h^{n - 1}\left(1 - h^{-1}\left(\dfrac{\rho}{1 - \vert x \vert^{2}}x - D\rho\right)\cdot D_{p}h\right)}.
  \end{equation*}

\end{proof}

\section{Conclusions}\label{Section 6}

\sloppy{}

In this paper, we studied the far field refractor problem with loss of energy in negative refractive index material. We first recalled the Snell law in vector form and derived the Fresnel formula in negative refractive index material. Then we proved the existence of the weak solution of the refraction problem with loss of energy in both the cases $\kappa < -1$ and $-1 < \kappa <0$. Finally, the inequality involving a Monge-Amp\`ere type operator satisfied by $\rho$ is derived. The conclusion of the existence of weak solutions is similar to that in \cite{GM13}. However, since $\kappa$ is negative in this paper, the process of the proof is different. Especially, for $\kappa < 0$, when proving the boundedness of some parameters, we use a different method from that in \cite{GM13} for the scaling of inequalities. For the inequality involving a Monge-Amp\`ere type operator satisfied by $\rho$, the form of the Eq (\ref{5.6}) is similar to the inequality (8.15) of the positive refractive index in \cite{GM13}, while the formula of $\omega$ is different. Besides, for $\kappa <0$, we need to take the absolute value of $\kappa$ in (\ref{5.6}). If we do not take the absolute value of $\kappa$, then the right side of (\ref{5.6}) becomes negative and the Monge-Amp\`ere type operator in (\ref{5.6}) is not elliptic, hence it is meaningless. Meanwhile, if we do not consider the loss of energy, then the energy condition is
\begin{equation*}
  \int_{\mathcal{T}_{\Gamma}(\omega)}f(x)dx = \int_{\omega}g(m)dm,
\end{equation*}
hence the Monge-Amp\`ere type equation satisfied by $\rho$ can be written as
\begin{equation*}
  \vert \det (D^{2}\rho + C^{-1}B) \vert = \frac{f(x)\vert \kappa \vert^{n - 2} \omega}{g(T(x))h^{n - 1}\left(1 - h^{-1}\left(\dfrac{\rho}{1 - \vert x \vert^{2}}x - D\rho\right)\cdot D_{p}h\right)}.
\end{equation*}
In fact, the inequality (\ref{5.6}) we have derived is universal. If we do not consider the loss of energy in negative refractive index material, we only need to take $t_{\Gamma}(x) = 1$ in (\ref{5.6}) and change the unequal sign to the equal sign, and if we consider the refraction problem in positive refractive index material with loss of energy, we only need to change the formula of $\omega$ and remove the absolute value of $\kappa$.

This paper used Minkowski method to solve the far field refractor problem with loss of energy in negative refractive index material, which is a remaining problem in \cite{St17}. Minkowski method is effective in solving the refractor problem with loss of energy. However, can the optimal transportation method be used to solve the refractor problem with loss of energy is still an open problem. Besides, in order to prove the existence of the weak solution of the far field refraction problem in negative refractive index material with loss of energy, we take $\varepsilon>0$ to strengthen Lemma \ref{lem2.1} to Remark \ref{rem2.1}. However, can we take $\varepsilon = 0$ in Remark \ref{rem2.1} is still an open problem as well.

\vspace{3mm}

\noindent {\bf Conflict of Interest}  { The authors declare no conflict of interest.}

\vspace{3mm}

\noindent {\bf Acknowledgments}  { The authors sincerely thank the anonymous referees for their valuable and insightful comments.
This work was supported by the National Natural Science Foundation of China (No. 12271093), the Jiangsu Provincial Scientific Research Center of Applied Mathematics (No. BK20233002), the Start-up Research Fund of Southeast University (No. 4007012503) and Shanghai Institute for Mathematics and Interdisciplinary Sciences (SIMIS) under grant number SIMIS-ID-2025-AD.}

\end{document}